\documentclass[11pt,reqno]{amsart}
\usepackage{amssymb,amsmath,amsfonts,amsthm,enumerate,stmaryrd}
\usepackage{mathrsfs}
\usepackage{nameref,hyperref,cleveref}
\usepackage{microtype}

\numberwithin{equation}{section}

\usepackage[left=.75 in, right=.75 in,top=.75 in, bottom=.75 in]{geometry}
\renewcommand{\Bar}{\overline}

\renewcommand{\S}{\mathcal{S}}
\newcommand{\R}{\mathbb{R}}
\newcommand{\N}{\mathbb{N}}
\newcommand{\Z}{\mathbb{Z}}

\newcommand{\C}{\mathbb{C}}
\renewcommand{\P}{\mathbb{P}}
\newcommand{\T}{\mathbb{T}}
\newcommand{\D}{\mathcal{D}}

\newcommand{\m}{\mathrm}

\newcommand{\lv}{\lVert}
\newcommand{\rv}{\rVert}

\newcommand{\J}{\boldsymbol{J}}

\newcommand{\al}{\alpha}
\newcommand{\be}{\beta}
\newcommand{\es}{\varnothing}
\newcommand{\lra}{\;\Leftrightarrow\;}
\newcommand{\ep}{\varepsilon}
\newcommand{\f}{\frac}

\newcommand{\gam}{\gamma}
\newcommand{\del}{\delta}

\newcommand{\pd}{\partial}

\newcommand{\grad}{\nabla}
\newcommand{\bpm}{\begin{pmatrix}}
\newcommand{\epm}{\end{pmatrix}}

\newcommand{\loc}{\m{loc}}
\renewcommand{\bar}{\overline}

\newcommand{\emb}{\hookrightarrow}

\newcommand{\norm}[1]{\left\lv#1\right\rv}
\newcommand{\abs}[1]{\left|#1\right|}
\newcommand{\p}[1]{\left(#1\right)}
\newcommand{\z}[1]{\mathring{#1}}
\newcommand{\br}[1]{\left\langle{#1}\right\rangle}

\renewcommand{\sb}[1]{\left[{#1}\right]}
\newcommand{\cb}[1]{\left\{{#1}\right\}}
\renewcommand{\bf}[1]{\mathbf{#1}}

\DeclareMathOperator{\curl}{curl}
\DeclareMathOperator{\Div}{div}
\DeclareMathOperator{\ten}{ten}

\providecommand{\br}[1]{\langle #1 \rangle}

\newtheorem{prop}{Proposition}[section]
\newtheorem{thm}[prop]{Theorem}
\newtheorem{defn}[prop]{Definition}
\newtheorem{lem}[prop]{Lemma}
\newtheorem{coro}[prop]{Corollary}
\newtheorem{innercustomthm}{Theorem}
\newenvironment{customthm}[1]
  {\renewcommand\theinnercustomthm{#1}\innercustomthm}
  {\endinnercustomthm}

\title[Global well-posedness near potential microflows]{
Analysis of micropolar fluids: existence of potential microflow solutions, nearby global well-posedness, and asymptotic stability
}

\author{Noah Stevenson}
\address{
Department of Mathematical Sciences\\
Carnegie Mellon University\\
Pittsburgh, PA 15213, USA
}
\email[N. Stevenson]{nwsteven@andrew.cmu.edu}

\author{Ian Tice}
\address{
Department of Mathematical Sciences\\
Carnegie Mellon University\\
Pittsburgh, PA 15213, USA
}
\email[I. Tice]{iantice@andrew.cmu.edu}
\thanks{I. Tice was supported by an NSF CAREER Grant (DMS \#1653161). N. Stevenson was supported by the summer research support provided by this grant. }



\begin{document}

\begin{abstract} 

In this paper we concern ourselves with an incompressible, viscous, isotropic, and periodic micropolar fluid. We find that in the absence of forcing and microtorquing there exists an infinite family of well-behaved solutions, which we call potential microflows, in which the fluid velocity vanishes identically, but the angular velocity of the microstructure is conservative and obeys a linear parabolic system.  We then prove that nearby each potential microflow, the nonlinear equations of motion are well-posed globally-in-time, and solutions are stable.  Finally, we prove that in the absence of force and microtorque, solutions decay exponentially, and in the presence of force and microtorque obeying certain conditions, solutions have quantifiable decay rates.

\end{abstract}

\maketitle


\section{Introduction}

\subsection{Overview}
The theory of micropolar fluids, first introduced by Eringen~\cite{eringen_1} to describe the mechanics of a microcontinuum, is an extension of the classical theory of fluid mechanics.  Among the novelties of the former theory are the effects of microstructure on the fluid.  In essence, the angular velocity and rotational inertia of the microstructure are accounted for at each point in the fluid, and the dynamics of these quantities couples to the bulk dynamics.  In the case of a viscous and incompressible fluid, the system is governed by a variant of the Navier-Stokes equations, coupled to dissipative evolution equations for the microangular momentum.  Micropolar fluids are common, and examples include: blood~\cite{ramkissoon_1,bhargava_1,mekheimer_1}, certain lubricants~\cite{allen_1,bayada_1,nicodemus_1,sinha_1}, liquid crystals~\cite{eringen_1,lhuillier_1,gay-balmaz_1}, and ferrofluids~\cite{nochetto_1}.

In  this paper we shall concern ourselves with the viscous and incompressible micropolar model.  For the sake of simplicity, our fluids are taken to be  spatially periodic, and our microstructure is assumed to be isotropic and homogeneous. Thus, our micropolar fluid occupies the three dimensional flat torus, $\T^3=\R^3/\Z^3$, and the microstructure has no preferred direction of rotation nor spatial dependence modulo proper rotation. The state of our micropolar fluid is described by three variables related via a system of nonlinear partial differential equations. The velocity and microangular velocity are a pair of evolving vector fields $u,\omega:\R^+\times\T^3\to\R^3$. The pressure, on the other hand, is an evolving scalar field $p:\R^+\times\T^3\to\R$.  The equations for such a fluid evolving from initial data $u_0,\omega_0:\T^3\to\R^3$, and subject to applied force and microtorque fields $f,g:\R^+\times\T^3\to\R^3$ read:
\begin{equation}\label{micro_polar}
\begin{cases}
\Div u=0 &\text{in }\R^+\times\T^3\\
\varrho( \pd_t u  + u \cdot \nabla u)  -\p{\ep + \frac{\kappa}{2}} \Delta u- \kappa \curl\omega + \grad p =f&\text{in }\R^+\times\T^3 \\
j(\pd_t\omega + u \cdot \nabla \omega)   -\p{\al+\gam}\Delta \omega- \p{\frac{\alpha}{3} + \beta - \gamma }   \grad\Div\omega+ 2\kappa \omega - \kappa \curl u =g &\text{in }\R^+\times\T^3\\
\p{u\p{0},\omega\p{0}}=\p{u_0,\omega_0}&\text{on } \T^3.
\end{cases}
\end{equation}
In the above, the physical parameters are as follows: $\varrho,j\in\R^+$ are the fluid density and microrotational inertia, $\ep,\al,\be,\gam\in\R^+$ are coefficients of viscosity and microviscosity, and $\kappa\in\R^+$ is the asymmetric viscosity coefficient. Note that $\kappa$ is what creates the coupling between the velocity and microangular velocity fields.

The appearance of these various coefficients in the problem \eqref{micro_polar} might appear awkward at first. Clarity can be found by rephrasing \eqref{micro_polar} in terms of stress tensors. First we need a quick definition: given a vector in $v\in \R^3$ we define the antisymmetric matrix $\ten(v)\in\R^{3\times3}$ as the unique one satisfying the identity $\ten(v)w=v\times w$ for all $w\in\R^3$. Then, if we define the stress and stress-couple tensors $S,C \in \R^{3 \times 3}$ via
\begin{equation}
    \begin{cases}
     S=-pI+\ep\p{Du+Du^{\m{t}}}+\kappa\p{\f12\p{Du-Du^{\m{t}}}-\ten(\omega)}\\
     C=\al\p{D\omega+D\omega^{\m{t}}-\f23\Div\omega I}+\be\Div\omega I+\gam\p{D\omega-D\omega^{\m{t}}},
    \end{cases}
\end{equation}
where the coefficients $\ep,\kappa,\al,\be$ and $\gamma$ appear more naturally, then we can rewrite system \eqref{micro_polar} as
\begin{equation}
    \begin{cases}
     \Div u=0&\text{in }\R^+\times\T^3\\
     \varrho\p{\pd_t+u\cdot\grad}u-\Div S=f&\text{in }\R^+\times\T^3\\
     j\p{\pd_t+u\cdot\grad}u-\Div C+\kappa\p{2\omega-\curl u}=g&\text{in }\R^+\times\T^3\\\p{u\p{0},\omega\p{0}}=\p{u_0,\omega_0}&\text{on }\T^3.
    \end{cases}
\end{equation}
Note that another one of the novelties of the micropolar model is that the stress tensor $S$, which is symmetric for standard fluids, has an antisymmetric contribution $\kappa \ten(\omega)$ due to the exchange of bulk angular momentum and microangular momentum.

Our first goal in this paper is to construct special solutions to \eqref{micro_polar}, which we call potential microflows.  These consist of solutions of the form $u=0$, $\omega = \zeta$, $p=0$, where $\zeta$ is a solution to a related parabolic problem.  The moniker potential microflow comes from the fact that $\zeta$ is a spatial gradient for all time.  The second goal of the paper is to construct global-in-time solutions to \eqref{micro_polar} near the potential microflows and to study their long time asymptotics.

\subsection{Some previous work}
Since the introduction of the micropolar model by Eringen~\cite{eringen_1} in the 1960s, the mathematics community has taken interest in the governing equations.  A full review of the math literature would be impractical, so we will only attempt a brief summary of the results related to the present paper.  We refer to the book of {\L}ukaszewicz~\cite{lukaszewicz_book} for some mathematical references.  Eringen's books~\cite{eringen2001microcontinuum,eringen_3} contain a wealth of references in the physical and engineering literature.  

The earliest result on global well-posedness of the micropolar equations in three dimensions is due to Galdi and Rionero \cite{galdi_1}.  {\L}ukaszewicz~\cite{lukaszewicz_4,lukaszewicz_3,lukaszewicz_2} studied global solutions and their asymptotics in three and two dimensions.  Rojas-Medar and Ortega-Torres~\cite{rojas-medar_orteg-torres_2005} and Yamaguchi~\cite{yamaguich_2005} constructed global solutions with other techniques.  Ferreira and Villamizar-Roa~\cite{ferreira_1} constructed distributional solutions.  Villamizar-Roa and Rodr\'{i}guez-Bellido~\cite{VR_RB_2008} constructed global solutions near stationary solutions and proved asymptotic stability.  Chen and Miao~\cite{chen_1} constructed solutions in critical Besov spaces.  

To the best of our knowledge, neither the potential microflow solutions nor the nearby solutions to \eqref{micro_polar} constructed in this paper have been studied previously.

\subsection{Statement of main results and discussion}

There are three main results in this paper.  For the sake of brevity we neglect to provide fully detailed statements here, and instead give informal abbreviated forms of the results.  The proper statements can be found later in the indicated theorems.

Our first result proves the existence of global-in-time potential microflow solutions to \eqref{micro_polar}.

\begin{customthm}{1}[Proved in Theorem \ref{micropotential_exist}]\label{result_1}
Given any conservative (curl-free) initial configuration for the angular velocity of the microstructure, $\zeta_0:\T^3\to\R^3$, there exists a microangular velocity $\zeta:\R^+\times\T^3\to\R^3$ such that the triple $\p{u,\omega,p}=\p{0,\zeta,0}$ is the unique solution to the system \eqref{micro_polar} with data $\p{u_0,\omega_0}=\p{0,\zeta_0}$ and vanishing force and microtorque, $f=g=0$.  These potential microflows are smooth in $\R^+ \times \T^3$ and decay to zero as $t \to \infty$. 
\end{customthm}

Our second result proves that nearby the potential microflow data the equations of motion in \eqref{micro_polar} are well-posed globally-in-time.

\begin{customthm}{2}[Proved in Theorem \ref{inverse function}]\label{result_2}
If $\zeta_0:\T^3\to\R^3$ is a sufficiently regular conservative generator for a potential microflow $\zeta$ as in Theorem \ref{result_1}, then the problem \eqref{micro_polar} is globally well-posed near $\zeta_0$ in the sense that to each tuple of data/forcing/microtorquing $\p{u_0,\omega_0,f,g}$ belonging to an open subset of an appropriate function space that contains $\p{0,\zeta_0,0,0}$, there exists a unique global-in-time solution $\p{u,\omega,p}$ to \eqref{micro_polar} that belongs to another appropriate function space.  Moreover, the map $\p{u_0,\omega_0,f,g}\mapsto\p{u,\omega,p}$ is smooth.
\end{customthm}

Our final result shows that nearby potential microflow data, the corresponding solutions to \eqref{micro_polar} are Lipschitz stable and gives sufficient conditions for attractiveness.

\begin{customthm}{3}[Proved in Theorem \ref{s_a}]\label{result_3}
If $\zeta_0:\T^3\to\R^3$ is a sufficiently regular conservative generator for a potential microflow $\zeta$ as in Theorem \ref{result_1}, then in a neighborhood of $\p{0,\zeta_0,0,0}$ in the space of data/forcing/microtorquing, the solutions to \eqref{micro_polar} are Lipschitz stable with respect to an appropriate norm.  Moreover, in the case we are given quantitative decay of the forcing and microtorquing, we find that the solutions satisfy corresponding quantitative decay estimates.
\end{customthm}

The potential microflows from Theorem \ref{result_1} are found as solutions to a time-dependent Lam\'e system, subject to the extra constraint that the field is conservative for all time.  In order to solve this system, we use the Leray projector (see Definition \ref{Leray}) to decouple the PDE into a pair of vectorial heat equations.  In Section \ref{sec_para_iso} we study these equations and record various standard parabolic estimates in terms of isomorphisms between certain Sobolev spaces.  These isomorphisms form the foundation of our subsequent analysis and also yield the existence of the desired potential microflow solutions in Theorem \ref{micropotential_exist}. 

With the potential microflow solutions in hand, we then turn to the construction of solutions to \eqref{micro_polar} nearby.  In order to control the behavior of the fluid velocity average for all time, we posit that a feature of the potential microflow forcing, $f=0$, persists in the general forcing; namely, we assume that $f$ has vanishing spatial average for all time.  This allows us to reduce, without loss of generality, to studying the problem \eqref{micro_polar} with the extra hypothesis that $u_0$ has vanishing spatial average.  We record the proof of this reduction in Appendix \ref{app_reduction}.

Our strategy for proving the global well-posedness assertion in Theorem \ref{result_2} is to pick appropriate container spaces for data/forcing/microtorquing and velocity/angular velocity/pressure and show that the natural nonlinear mapping from the latter to the former induced by \eqref{micro_polar} is smooth and a local diffeomorphism near each potential microflow.  Naturally, this is verified using the inverse function theorem.  This leads us to study the linearization of \eqref{micro_polar} around a potential microflow in Section \ref{sec_linearization}.  The resulting linearization (see \eqref{mp_linearized}) is a coupled vectorial parabolic system with coefficients that have nontrivial space-time dependence.  This precludes the use of semigroup techniques to construct solutions.  Instead, we take advantage of our analysis in Section \ref{sec_para_iso} and construct solutions with the help of the isomorphisms developed there and the Banach fixed point theorem.  

The fixed point scheme is most conveniently realized in the context low temporal regularity and integrability, and it is in this setting that we construct solutions.  We then exploit the natural energy-dissipation structure associated to \eqref{micro_polar} and its linearization \eqref{mp_linearized} to begin a bootstrap argument that ultimately shows that these solutions enjoy better temporal regularity and integrability.  With these results in hand, we then prove in Theorem \ref{linear iso} that the linearization \eqref{mp_linearized} induces an isomorphism between appropriate spaces.

In Section \ref{sec_nonlinear} we present the nonlinear analysis of \eqref{micro_polar}.  We employ some analysis tools from Appendix \ref{app_tools} to show that the nonlinear map associated to \eqref{micro_polar} is smooth on the spaces from Theorem \ref{linear iso}.  The theorem then shows that the linearization around a potential microflow is a linear homeomorphism, and so in Theorem \ref{inverse function} we employ the inverse function theorem to produce solutions to \eqref{micro_polar} near the potential microflows.  

Finally, we prove Theorem \ref{s_a}, which establishes the stability and attractiveness assertions of Theorem \ref{result_3}.  The proof relies on a synthesis of estimates provided by the inverse function theorem and by the energy-dissipation structure of \eqref{micro_polar}.

\subsection{Conventions of notation}

Here we record notation used throughout the paper.  The naturals, integers, reals, and complex numbers are denoted $\N$, $\Z$, $\R$, $\C$, respectively. We assume that $0\in\N$ and write $\N^+$ for the set $\N\setminus\cb{0}$.   Similarly, $\R^+ = (0,\infty)$ is the interval of positive real numbers. $\Bar{\R}$ is the usual two-point compactification of $\R$, created by adding the lower and upper endpoints of $\mp\infty$.   We say a constant $C$ is universal if it depends on the various physical parameters appearing in \eqref{micro_polar}, the physical dimension, or regularity parameters.  

Next we recall how distributions work on the torus.
\begin{defn}[Test functions and distributions]\label{space of test functions}
Let $d,\ell \in \N^+$.
\begin{enumerate}
\item
For each $m\in\N$ we define a seminorm $[\cdot]_m:C^\infty\p{\T^d;\C}\to\R$ via  
\begin{equation}
 [f]_m = \sum_{\substack{\al\in\N^d\\\abs{\alpha}=m}} \sup_{x \in \T^d}  \abs{\partial^\alpha f(x)}.
\end{equation}
This countable family of seminorms induces a Fr\'echet vector topology on $C^\infty\p{\T^d;\C}$.  When equipping $C^\infty\p{\T^d;\C}$ with this topology we shall use the notation $\D\p{\T^d;\C}$ and refer to this space as the space of test functions.

\item We denote the space of linear and continuous mappings $\D\p{\T^d;\C}\to\C^\ell$ by $\D^\ast\p{\T^d;\C^\ell}$ and refer to this set as the $\C^\ell$-valued periodic distributions.

\item We shall use the standard bracket notation to denote the pairing between the spaces of distributions and test functions, i.e. we define
\begin{equation}
\br{\cdot,\cdot}:\mathcal{D}^\ast\p{\T^d;\C^\ell}\times\mathcal{D}\p{\T^d;\C}\to\C^\ell \text{ via } \br{T,\psi}=T\p{\psi}.
\end{equation}

\end{enumerate}
\end{defn}

Next we recall the distributional Fourier transform.

\begin{defn}[Complex exponentials and the Fourier transform]\label{complex exp}
Let $d,\ell \in \N^+$.  
\begin{enumerate}
    \item For each $k\in\Z^d$ we define $\bf{e}_k:\T^d\to\C$ via $\bf{e}_k\p{x}= e^{2\pi i k \cdot x}$.  Clearly,  $\bf{e}_k\in \D\p{\T^d;\C}$ for each $k$.
    \item We define $\hat{\cdot}:\D^\ast\p{\T^d;\C^\ell} \to \p{\C^\ell}^{\Z^d}$ via $\hat{T}\p{k}=\br{T, \bf{e}_{-k}}$ for $k\in\Z^d$ and $T\in\D^\ast\p{\T^d;\C^\ell}$.  This mapping is called the  Fourier transform.
\end{enumerate}
\end{defn}

The decay at infinity of the Fourier transform of some distribution encodes regularity properties of the distribution. The following family of spaces exploits this fact.
\begin{defn}[Spatial Sobolev spaces]\label{sobolev spaces}
Let $s\in\R$, $d,\ell\in\R^+$.
\begin{enumerate}
    \item The $\C^\ell$-valued spatial Sobolev space is the set
    \begin{equation}
        H^s\p{\T^d;\C^\ell}=\cb{T\in\mathcal{D}^\ast\p{\T^d;\C^\ell}\;|\;\sum_{k\in\Z^3}\p{1+\abs{k}^2}^{s/2}\abs{\hat{T}\p{k}}^2<\infty},
    \end{equation}
    equipped with the inner-product
    \begin{equation}
        \p{T_0,T_1}_{H^s} = \sum_{k\in\Z^3}\p{1+\abs{k}^2}^{s}\p{\hat{T_0}\p{k} , \hat{T_1}\p{k}}_{\C^\ell}
    \end{equation}
    and norm $\norm{\cdot}_{H^s}$, generated by the above inner-product. It is well known that $H^s\p{\T^d;\C^\ell}$ is a separable Hilbert space.
    \item The $\R^\ell$-valued spatial Sobolev space is the closed subspace
    \begin{equation}
        H^s\p{\T^d;\R^\ell}=\cb{T\in H^s\p{\T^d;\C^\ell}\;|\;T=\Bar{T}}.
    \end{equation}
    We recall that the complex conjugate of a distribution $T\in\mathcal{D}^\ast\p{\T^d;\C^\ell}$ is also a distribution $\bar{T}\in\mathcal{D}^\ast\p{\T^d;\C^\ell}$ with action on $\psi\in\mathcal{D}\p{\T^d;\C}$ via $\br{\Bar{T},\psi}=\Bar{\br{T,\Bar{\psi}}}$.
    \item For $\mathbb{K}=\R$ or $\C$, the $\mathbb{K}^\ell$-valued and mean zero spatial Sobolev space is the closed subspace
    \begin{equation}
        \z{H}^s\p{\T^d;\mathbb{K}^\ell}=\cb{T\in H^s\p{\T^d;\mathbb{K}^\ell}\;|\;\hat{T}\p{0}=0}.
    \end{equation}
    We equip this space with a slightly modified inner-product
    \begin{equation}
        \p{T_0,T_1}_{\z{H}^s}=\sum_{k\in\Z^d\setminus\cb{0}}\abs{k}^{2s}\p{\hat{T_0}\p{k},\hat{T_1}\p{k}}_{\mathbb{C}^\ell}
    \end{equation}
    and let $\norm{\cdot}_{\z{H}^s}$ denote the corresponding norm, which is equivalent to the usual one.
\end{enumerate}

\end{defn}

Next we recall the a useful pseudo-differential operator, which acts on distributions in the spatial Sobolev spaces.

\begin{defn}\label{psuedo diff}
If $r,s\in\R$ and $d,\ell\in\N^+$ we define the operator $\J^s :H^r\p{\T^d;\C^\ell}\to H^{r-s}\p{\T^d;\C^\ell}$ via 
\begin{equation} 
\J^sT=\sum_{k\in\Z^d}\p{1+\abs{k}^2}^{s/2}\hat{T}\p{k}\bf{e}_k.
\end{equation}
Note that for $\mathbb{K}=\C$ or $\R$ $\J^s$ is an isometric isomorphism from  $H^r\p{\T^d;\mathbb{K}^\ell}$ to $H^{r-s}\p{\T^d;\mathbb{K}^\ell}$.
\end{defn}

Next, we wish to decompose the real spatial Sobolev spaces as orthogonal direct sum of solenoidal and conservative vector fields.

\begin{defn}[Leray projection]\label{Leray}
If $s\in\R$ and $d\in\R^+$ we define an operator, called the Leray projector, $\P:H^s\p{\T^d;\R^d}\to H^s\p{\T^d;\R^d}$ via
\begin{equation} 
\P T=\hat{T}(0)+\sum_{k\in\Z^d\setminus\cb{0}}\p{I-\f{k\otimes k}{\abs{k}^2}}\hat{T}\p{k}\bf{e}_k.
\end{equation}
It can be shown that $\P$ is well-defined (see Lemma \ref{real_lemma}), self-adjoint in  $H^s\p{\T^d;\R^d}$, and satisfies $\P^2=\P$. Therefore, by basic Hilbert space theory, it must be a projection onto its image. Moreover, $H^s\p{\T^d;\R^d}$ can be realized as the orthogonal direct sum of the image and kernel of $\P$. Indeed, we define $\P H^s\p{\T^d;\R^d}=H^s_{\perp}\p{\T^d;\R^d}$ and $\p{I-\P}H^s\p{\T^d;\R^d}=H^s_{\|}\p{\T^d;\R^d}$, which yields the decomposition
\begin{equation}\
H^s\p{\T^d;\R^d}=H^s_\perp\p{\T^d;\R^d}\oplus H^s_\|\p{\T^d;\R^d}.
\end{equation}
This notation indicates that if $T\in H^s_\perp\p{\T^d;\R^d}$, then $\hat{T}\p{k}\in\C^d$ points orthogonally (in the $\C^d$ sense) to $k\in\Z^d$. On the other hand, if $T\in H^s_{\|}\p{\T^d;\R^d}$, then $\hat{T}\p{k}$ is parallel (in the $\C^d$ sense) to $k$ (and, in particular, $T\in\z{H}^s\p{\T^d;\R^d}$). We shall refer to $H^s_\perp\p{\T^d;\R^d}$ as being a space of solenoidal fields, since each member of this space has trivial distributional divergence.  Similarly, each member of $H^s_{\|}\p{\T^d;\R^d}$ is said to be conservative, as they are distributional gradients (see Proposition \ref{potential_map}).
\end{defn}

Finally, we need some notions of the theory of (infinite dimensional) Banach-valued Sobolev spaces on subsets of the real line.

\begin{defn}[Space-time Sobolev spaces]\label{time into banach}
Let $s\in\R$, $d,\ell\in\N^+$, $n\in\N$, $\mathbb{K}=\C$ or $\R$, $\mathcal{X}\subseteq H^s\p{\T^d;\mathbb{K}^\ell}$ be a closed subspace, and $\es\neq I\subseteq\R$ be an open set.

\begin{enumerate}
    \item  If $f,g\in L^1_{\loc}\p{I;\mathcal{X}}$, we say that $g$ is the $n^{\m{th}}$ weak derivative of $f$, if for all $\psi\in C^\infty\p{I;\mathbb{K}}$ it holds that \begin{equation}\int_{I}f\psi^{\p{n}}=\p{-1}^n\int_{I}g\psi.\end{equation}
    In this case we write $g=f^{\p{n}}$, and when $n=1$ we write $g=f'$.
    \item The temporal Sobolev space or order $n$ is the set
    \begin{equation}
        H^n\p{I;\mathcal{X}}=\cb{f\in L^1_\loc\p{I;\mathcal{X}}\;:\;\forall\;j\in\cb{0,1,\dots,n},\;f^{\p{j}}\;\text{exists in the weak sense, and}\;\int_{I}\norm{f^{\p{n}}}_{\mathcal{X}}^{2}<\infty},
    \end{equation}
    equipped with inner-product
    \begin{equation}
        \p{f_0,f_1}_{H^n\p{I;\mathcal{X}}}=\sum_{j=0}^n\int_{I}\p{f_0^{\p{j}}, f_1^{\p{j}}}_{\mathcal{X}}.
    \end{equation}
\end{enumerate}

\end{defn}

\section{Parabolic isomorphisms and construction of the potential microflows}\label{sec_para_iso}

The first goal of this section is to construct solution operators to two different linear parabolic equations between the spaces from Definitions \ref{sobolev spaces}, \ref{Leray}, and \ref{time into banach}.  With these in hand, we then present the construction of the potential microflow solutions to \eqref{micro_polar}.

\subsection{Vectorial heat flow }

In this subsection we are interested in solutions $u: I \times \T^3 \to \R^3$  to the initial value problem
\begin{equation}\label{pde_1}
\begin{cases}
\pd_t u-\al \Delta u=f &\text{in }I\times\T^3\\
u\p{0,\cdot}=u_0&\text{on } \T^3,
\end{cases}
\end{equation}
where $I = (0,T)$ for $T \in (0,\infty]$, $\alpha \in \R^+$, and we are given initial data $u_0 \in H^{1+r}(\T^3;\R^3)$ and average-zero forcing $f\in L^2(I;\z{H}^{r}(\T^3;\R^3))\cap H^n(I;\z{H}^{r-2n}(\T^3;\R^3))$ for $r \in \R$ and $n \in \N$.  Of course, \eqref{pde_1} is a vectorial variant of the standard heat equation, and most of what we present here is well-known.  Our goal here is thus to quickly present the main features of \eqref{pde_1} in a functional analytic form useful for our subsequent work in the paper.  

We begin by defining the solution operator for \eqref{pde_1}.

\begin{defn}\label{Heat flow 1}
Let $\al\in\R^+$ and $r \in \R$.  We define the $\alpha-$heat flow mapping $\S_\alpha : \z{H}^r(\T^3;\R^3) \to L^1_{\m{loc}}(\R^+;\z{H}^r(\T^3;\R^3))$ via 
\begin{equation} 
\S_\al\p{g}\p{t}=\sum_{k\in\Z^3}\exp\p{-4\pi^2\al\abs{k}^2t}\hat{g}\p{k}\bf{e}_k,
\end{equation}
where the series clearly converges in $\z{H}^r(\T^3;\R^3)$ for each $t >0$ and defines an $\R^3$-valued distribution thanks to Lemma \ref{real_lemma}.
\end{defn}

The next result establishes the essential properties of the map $\S_\alpha$.

\begin{prop}\label{heat flow 1}
Let $\al\in\R^+$, $r \in \R$, and consider the mapping $\S_\alpha$ from Definition \ref{heat flow 1}.  Then the following hold.
\begin{enumerate}
\item If $s\in\R$ with $r\le s$, then there exists a constant $C >0$, depending on $\alpha$, $r$, and $s$, such that
\begin{equation} 
\norm{\S_\al\p{g}(t)}_{\z{H}^s} \le C \f{\exp\p{-2\pi^2\al t}}{t^{\f{s-r}{2}}}\norm{g}_{\z{H}^r}
\end{equation}
for every $g\in\z{H}^r\p{\T^3;\R^3}$ and $t \in \R^+$.  In particular, $\S_\al$ is smoothing: if  $g\in\z{H}^r\p{\T^3;\R^3}$, then $\S_\alpha\p{g}(t) \in C^\infty(\T^3;\R^3)$ for every $t \in \R^+$.

\item If $n\in\N$ and $0<T\in\bar{\R}$,  then
\begin{equation} 
\S_\al:\z{H}^{1+r}\p{\T^3;\R^3}\to L^2\p{\p{0,T};\z{H}^{2+r}\p{\T^3;\R^3}}\cap H^n\p{\p{0,T};\z{H}^{2+r-2n}\p{\T^3;\R^3}}
\end{equation}
is a bounded linear mapping.

\item  If $g \in \z{H}^r_\perp\p{\T^3;\R^3}$, then for each $t\in\R^+$ we have $\mathcal{S}_\al\p{g}\p{t} \in \z{H}^r_\perp\p{\T^3;\R^3}$.  Similarly, if $g\in H^r_\|\p{\T^3;\R^3}$, then for each $t\in\R^+$ we have $\mathcal{S}_\al\p{g}\p{t} \in H^r_\|\p{\T^3;\R^3}$.

\item For any $g\in\z{H}^r\p{\T^3;\R^3}$ we have 
\begin{equation}
\lim_{t\to0^+}\norm{\S_\al\p{g}\p{t}-g}_{\z{H}^r}=0 
\end{equation}
and 
\begin{equation}
 \partial_t \S_\alpha\p{g}(t) = \alpha \Delta \S_\alpha \p{g}(t) \text{ for } t \in \R^+.
\end{equation}

\end{enumerate}
\end{prop}

\begin{proof}

Let $s \in \R$ with $r \le s$.  We then compute
\begin{equation}
\begin{split}
\norm{\mathcal{S}_\al(g)(t)}^2_{\z{H}^{s}}&=\sum_{k\in\Z^3\setminus\left\{0\right\}}|k|^{2s}\exp\left(-8\pi^2\al|k|^2t\right)\abs{\hat g(k)}^2\\
&\le \sup_{\rho\in\Z^3\setminus\left\{0\right\}}\left\{|\rho|^{2(s-r)}\exp\left(-4\pi^2\al|\rho|^2t\right)\right\} \sum_{k\in\Z^3\setminus\left\{0\right\}}\left|k\right|^{2r}\exp\left(-4\pi^2\al|k|^2t\right)\abs{\hat g(k)}^2\\
& \le C \f{\exp\left(-4\pi^2\al t\right)}{t^{s-r}}\norm{g}^2_{\z{H}^r}
\end{split}
\end{equation}
for the constant
\begin{equation}
 C = \p{4\pi^2\al}^{r-s}\sup\cb{\abs{\tilde{\rho}}^{2\p{s-r}}\exp\p{-\abs{\tilde{\rho}}^2}\;|\;\tilde{\rho}\in\R^3} \in\R^+.
\end{equation}
This constant depends only on $r,s$, and $\alpha$, which proves the first item.

We now turn to the proof of the second item.  We will present only a sketch in the case $T=\infty$, as the case $T < \infty$ follow directly from this.  Let $g\in\z{H}^{1+r}(\T^3;\R^3)$.   We first use Tonelli's theorem to compute
\begin{multline}
	\norm{\mathcal{S}_\al(g)}_{L^2\z{H}^{2+r}}^2 =\int_{\R^{+}}\norm{\mathcal{S}_\al\left(g \right)(t)}^2_{\z{H}^{2+r}}\m{d}t 
	=\int_{\R^{+}}\sum_{k\in\Z^3\setminus\{0\}}|k|^{2r+4}\exp\left(-8\pi^2\al|k|^2t\right)\abs{\hat{g}(k)}^2\m{d}t \\
	=\frac{1}{8\pi^2 \al}  \sum_{k\in\Z^3\setminus\{0\}}|k|^{2r+2} \abs{\hat{g}(k)}^2 
	=\f{1}{8\pi^2\al}\norm{g}^2_{\z{H}^{1+r}},
\end{multline}
which shows that $\S_\al(g) \in L^2(\R^+; \z{H}^{2+r}\p{\T^3;\R^3})$.  Next we note that a direct computation, which we omit for the sake of brevity, shows that if $\S_\alpha(g)^{\p{k}}$ denotes the $k^{th}$ weak partial time derivative of $\S_{\alpha}(g)$, then
\begin{equation}
 \S_\al(g)^{\p{k}}(t)=(\alpha \Delta)^k \S_\al(g)(t) \text{ for } t \in \R^+.
\end{equation}
From this we then deduce that
\begin{equation} 
\norm{\S_\al(g)^{\p{k}}}^2_{L^2\z{H}^{2+r-2k}}=\left(4\pi^2\al\right)^k\norm{\mathcal{S}_\al(g)}_{L^2\z{H}^{2+r}}^2=\f12\left(4\pi^2\al\right)^{k-1}\norm{g}_{\z{H}^{1+r}}^2
\end{equation}
for each $0 \le k \le n$.  This prove that $\S_\al(g) \in H^n(\R^+; \z{H}^{2+r-2n} \p{\T^3;\R^3})$, which completes the proof of the second item.

The third item follows directly from the definition of $\S_\al$, and the fourth item follows from the dominated convergence theorem and the above calculations.
\end{proof}

Next we wish to define `convolution' with $\S_\al$ as a solution operator to the inhomogeneous problem \eqref{pde_1} with zero initial data.

\begin{defn}\label{Convolution}
For $\p{r,n}\in\R\times\N$ and $0< T\in\bar{\R}$ we define
\begin{multline} 
\S_\al\ast:L^2\p{\p{0,T};\z{H}^r\p{\T^3;\R^3}}\cap H^n\p{\p{0,T};\z{H}^{r-2n}\p{\T^3;\R^3}} \\
\to L^2\p{\p{0,T};\z{H}^{2+r}\p{\T^3;\R^3}}\cap H^{n+1}\p{\p{0,T};\z{H}^{r-2n}\p{\T^3;\R^3}}
\end{multline}
via
\begin{equation} 
\S_\al\ast f\p{t}=\sum_{k\in\Z^3}\bf{e}_k\int_{\p{0,t}}\exp\p{-4\pi^2\al\abs{k}^2\p{t-\tau}}\hat{f}\p{\tau,k}\;\m{d}\tau 
\end{equation}
for $t\in\p{0,T}$.  Note that this defines an $\R^3$-valued distribution by Lemma \ref{real_lemma}.
\end{defn}

It's not a priori clear that $\S_\al\ast$ actually takes values in the codomain listed in Definition \ref{Convolution}.  We verify this and some other basic properties now.

\begin{prop}\label{Convolution well defined}
Let $\p{r,n}\in\R\times\N$, $0< T\in\bar{\R}$, and consider $\S_\al\ast$ given by Definition \ref{Convolution}.  Then $\S_\al\ast$ is well-defined with the codomain stated in the definition.  Moreover, $\S_\al\ast$ is a bounded linear map, and for any $f\in L^2\p{\p{0,T};\z{H}^r(\T^3;\R^3)}\cap H^n\p{\p{0,T};\z{H}^{r-2n}(\T^3;\R^3)}$ we have that 
\begin{equation}\label{Convolution well defined limit}
\lim_{t\to0^+}\norm{\S_\al\ast f\p{t}}_{\z{H}^{1+r}}=0 
\end{equation}
and 
\begin{equation}
 \partial_t \S_\alpha\ast f(t) = \alpha \Delta \S_\alpha\ast f(t) + f(t) \text{ for } t \in (0,T).
\end{equation}

\end{prop}	
\begin{proof}
Linearity of $\S_\al\ast$ is clear, so it suffices to prove boundedness.  Let $f\in L^2\p{\p{0,T};\z{H}^r(\T^3;\R^3)}\cap H^n\p{\p{0,T};\z{H}^{r-2n}(\T^3;\R^3)}$.  Using Tonelli's theorem and Young's convolution inequality, we compute
\begin{equation}
\begin{split}
\norm{\S_\al\ast f}^2_{L^2\p{\p{0,T};\z{H}^{2+r}}}& 
\le\sum_{k\in\Z^3\setminus\{0\}}|k|^{2r+4}\int_{(0,T)}\p{\int_{(0,T)}\exp\left(-4\pi^2\al|k|^2(t-\tau)\right)\abs{\hat{f}(\tau,k)}\m{d}\tau}^2\;\m{d}t\\
&\le\sum_{k\in\Z^3\setminus\left\{0\right\}}|k|^{2r+4}\p{\int_{(0,T)}\exp\p{-4\pi^2\al|k|^2t}\m{d}t}^2\int_{(0,T)}\abs{\hat f(\tau,k)}^2\m{d}\tau \\
& =\f{1}{4\pi^2\al}\sum_{k\in\Z^3\setminus\left\{0\right\}}\abs{k}^{2r}\int_{(0,T)}\abs{\hat f(\tau,k)}^2\m{d}\tau\\
&=\f{1}{4\pi^2\al}\norm{f}^2_{L^2\p{\p{0,T};\z{H}^r}}.
\end{split}
\end{equation}

On the other hand, a simple computation shows that for all $\phi\in C^\infty_c\p{\p{0,T};\C}$ we have
\begin{equation} 
-\int_{\p{0,T}}\phi'\;\S_\al\ast f=\int_{\p{0,T}}\p{\al\Delta \S_\al\ast f+f}\phi
\end{equation}
and hence $\S_\al\ast f$ is weakly differentiable in time with  $\p{\S_\al\ast f}'=\al \Delta \S_\al\ast f+f$.  Arguing as above, we deduce from this identity that 
\begin{equation} 
\norm{\p{\mathcal{S}_\al\ast f}'}_{L^2\p{\p{0,T}\z{H}^r}}\le C\norm{f}_{L^2\p{\p{0,T};\z{H}^{r}}}.
\end{equation}
Iterating and applying Proposition \ref{interpolation}, we deduce the existence of a constant $C>0$, depending on $s,r,n,\alpha$, such that
\begin{equation} 
\norm{\S_\al\ast f}_{L^2\p{\p{0,T};\z{H}^{2+r}}}+\norm{\S_\al\ast f}_{H^{n+1}\p{\p{0,T};\z{H}^{r-2n}}}\le C\p{\norm{f}_{L^2\p{\p{0,T};\z{H}^{r}}}+\norm{f}_{H^n\p{\p{0,T};\z{H}^{r-2n}}}}.
\end{equation}
This proves that $\S_\al\ast$ is well-defined and gives rise to a bounded linear map.

It remains to prove \eqref{Convolution well defined limit}.  For this we use Tonelli's theorem and the Cauchy-Schwarz inequality to bound, for $t \in \R^+$, 
\begin{multline}
\norm{\S_\al\ast f(t)}_{\z{H}^{1+r}}^2 \le \sum_{k \in \Z^3 \backslash \{0\}} \abs{k}^{2r+2} \p{\int_{\R^+} \exp(-8\pi^2 \alpha \abs{k}^2 \tau \m{d}\tau ) } \p{ \int_{(0,t)} \abs{\hat{f}(\tau,k)}^2  \m{d}\tau } \\
= \frac{1}{8\pi^2 \alpha} \norm{f}_{L^2\p{\p{0,t};\z{H}^r}}^2.
\end{multline}
Then \eqref{Convolution well defined limit} follows from this and the monotone convergence theorem.
\end{proof}

The data flow operator from Definition \ref{Heat flow 1} and the convolution operator from Definition \ref{Convolution} sum to create a solution operator to \eqref{pde_1}.  We prove this now.

\begin{thm}\label{type 1 solution}
Let $\p{r,n}\in\R\times\N$, $0<T\in\bar{\R}$, and consider the mapping
\begin{multline} 
\Upsilon:L^2\p{\p{0,T};\z{H}^{2+r}\p{\T^3;\R^3} }\cap H^{n+1}\p{\p{0,T};\z{H}^{r-2n}\p{\T^3;\R^3}} \\
\to \z{H}^{1+r}\p{\T^3;\R^3} \times\p{L^2\p{\p{0,T};\z{H}^r\p{\T^3;\R^3}}\cap H^n\p{\p{0,T};\z{H}^{r-2n}\p{\T^3;\R^3}}}
\end{multline}
defined via 
 \begin{equation} 
\Upsilon\p{u}=\p{u\p{0},\partial_t u-\al\Delta u},
\end{equation}
where $u\p{0}$ is understood as $\lim_{t\to 0^+}u\p{t}$ in the $\z{H}^{1+r}\p{\T^3;\R^3}$ topology (see Proposition \ref{interpolation}). Then the following hold.
\begin{enumerate}
 \item$\Upsilon$ is well-defined and a linear isomorphism.
 \item The inverse of $\Upsilon$ is given explicitly as $\Upsilon^{-1}\p{g,f}=\S_\al g+\S_\al\ast f$ for $\S_\al$ and $\S_\al\ast$ as defined in Definitions \ref{Heat flow 1} and \ref{Convolution}, respectively.
\end{enumerate}
\end{thm}

\begin{proof}
Proposition \ref{interpolation} ensures us that if $u$ belongs to the domain of $\Upsilon$, then after modification on a null set $u\in UC^0_b\p{\p{0,T};\z{H}^{1+r}\p{\T^3;\R^3}}$. Thus there exists $u\p{0}\in\z{H}^{1+r}\p{\T^3;\R^3}$ for which $\lim_{t\to0^+}u\p{t}=u\p{0}$ in the $\z{H}^{1+r}\p{\T^3;\C^3}$ topology. Moreover, there is a constant $C\in\R^+$ depending only on $T$ and $\al$ for which we may estimate:
\begin{equation} 
\norm{u\p{0}}_{\z{H}^{1+r}}\le\sup_{t\in\p{0,T}}\norm{u\p{t}}_{\z{H}^{1+r}}\le C\p{\norm{u}_{L^2\p{\p{0,T};\z{H}^{2+r}}}+\norm{u}_{H^{n+1}\p{\p{0,T};\z{H}^{r-2n}}}}
\end{equation}
and 
\begin{equation}
 \norm{\partial_t u-\al\Delta u}_{L^2\p{\p{0,T};\z{H}^{r}}}+\norm{\partial_t u-\al\Delta u}_{H^n\p{\p{0,T};\z{H}^{r-2n}}}\le C\p{ \norm{u}_{L^2\p{\p{0,T};\z{H}^{2+r}}}+\norm{u}_{H^{n+1}\p{\p{0,T};\z{H}^{r-2n}}}}.
\end{equation}
Hence $\Upsilon$ is well defined and bounded. 

We now turn to the proof that $\Upsilon$ is surjective. Propositions \ref{heat flow 1} and \ref{Convolution well defined} ensure us that for a pair $\p{g,f}$ belonging to the codomain of $\Upsilon$ we have that $\S_\al g+\S_\al\ast f$ belongs to the domain of $\Upsilon$, $\lim_{t\to0^+}\p{\S_\al g+\S_\al\ast f}=g$ in $\z{H}^{1+r}\p{\T^3;\R^3}$, and 
\begin{equation} 
\partial_t \p{S_\al g+\S_\al\ast f} = \p{S_\al g+\S_\al\ast f}'= \al\Delta\p{\S_\al g+\S_\al\ast f}+f.
\end{equation}
Together, these imply that $\Upsilon\p{\S_\al g+\S_\al\ast f}=\p{g,f}$, which establishes surjectivity. 

We now turn to injectivity.  Suppose that $u$ is in the kernel of $\Upsilon$.   Then for all almost every $t\in\p{0,T}$ we have $\partial_t u\p{t}- \al\Delta u\p{t}=0$.  We take the inner-product with $u\p{t}$ in the space $\z{H}^r\p{\T^3;\R^3}$ in order to observe that for almost every $t\in\p{0,T}$, 
\begin{equation} 
\f12\p{\norm{u\p{t}}_{\z{H}^r}^2}'+4\pi^2\al\norm{u\p{t}}_{\z{H}^{1+r}}^2=0.
\end{equation}
Integrating and using the fact that $\lim_{t\to0^+}\norm{u\p{t}}_{\z{H}^{1+r}}=0$, we deduce that $u =0$.  Hence $\Upsilon$ is injective, and thus an isomorphism.
\end{proof}

The final result in this subsection concerns the restriction of $\Upsilon$ from Theorem \ref{type 1 solution} to spaces of the form $\z{H}^r_{\Gamma}\p{\T^3;\R^3}$, where $\Gamma \in \{\perp,\|\}$, as defined in Definition \ref{Leray}.

\begin{thm}\label{1 Preservation}
Let $\p{r,n}\in\R\times\N$ and $0<T\in\bar{\R}$.  For $\Gamma \in \{\perp,\|\}$  define
\begin{multline}
 \Upsilon_\Gamma : \p{ L^2\p{\p{0,T};\z{H}^{2+r}_\Gamma \p{\T^3;\R^3}}\cap H^{n+1}\p{\p{0,T};\z{H}^{r-2n}_\Gamma \p{\T^3;\R^3}}} \\
 \to \z{H}^{1+r}_\Gamma \p{\T^3;\R^3} \times L^2\p{\p{0,T};\z{H}^r_\Gamma \p{\T^3;\R^3} }\cap H^n\p{\p{0,T};\z{H}^{r-2n}_\Gamma \p{\T^3;\R^3}}
\end{multline}
via 
\begin{equation}
 \Upsilon_\Gamma(u) = (u(0), \partial_t u - \alpha \Delta u) = \Upsilon(u).
\end{equation}
Then $\Upsilon_\Gamma$ are well-defined, bounded, and linear isomorphisms.
\end{thm}
\begin{proof}
This follows directly from the fact that the Leray projector $\P$, as given in Definition \ref{Leray}, commutes with the isomorphism $\Upsilon$ from Theorem \ref{type 1 solution}.
\end{proof}

\subsection{Time-dependent Lam\'{e} system }

We now turn our attention to solutions $\omega: I \times \T^3 \to \R^3$ of the time-dependent Lam\'{e} system
\begin{equation}\label{pde_2}
\begin{cases}
\pd_t\omega-\p{\al+\gam}\Delta\omega-\p{\be-\gam}\grad\Div\omega+\del\omega=f&\text{in } I\times\T^3\\
\omega\p{0,\cdot}=\omega_0&\text{on }\T^3,
\end{cases}
\end{equation}
where $I = (0,T)$ for $T \in (0,\infty]$, $\al,\be\in\R^+$, $\gam\in\R^+\cup\cb{0}$, and we are given initial data $\omega_0 \in H^{1+r}(\T^3;\R^3)$ and forcing $f\in L^2(I;H^{r}(\T^3;\R^3))\cap H^n(I;H^{r-2n}(\T^3;\R^3))$ for $r \in \R$ and $n \in \N$.   As in the previous subsection, we intend to show that \eqref{pde_2} induces an isomorphism between appropriate spaces. The analysis is sufficiently similar to that associated to the problem \eqref{pde_1} that we omit most of it here.  The main difference between \eqref{pde_1} and \eqref{pde_2} is the nontrivial evolution of the spatial averages for the latter, which warrants recording the following results separately.

We begin by observing how \eqref{pde_2} interacts with the Leray projector, at least for smooth solutions.  This motivates our strategy for solving \eqref{pde_2}.

\begin{lem}\label{decoupling_lemma}
Suppose that $\omega,f\in C^\infty\p{I\times\T^3;\R^3}$ and define $\P\omega=\omega_\perp$, $\P f=f_\perp$ $\p{I-\P}\omega=\omega_\|$, $\p{I-\P}f=f_\|$.  Then $\omega$ and $f$ satisfy the first equation of \eqref{pde_2} if and only if 
\begin{equation}\label{decoupling_lemma_0} 
\begin{cases}
\pd_t\omega_\perp-\p{\al+\gam}\Delta\omega_\perp+\del\omega_\perp=f_\perp &\text{in } I\times\T^3 \\
\pd_t\omega_\|-\p{\al+\be}\Delta\omega_\|+\del\omega_\|=f_\| &\text{in } I\times\T^3.
\end{cases}
\end{equation}
\end{lem}
\begin{proof}
The symbol of the differential operator $-\p{\al+\gam}\Delta-\p{\be-\gam}\grad\Div+\del$ is
\begin{equation} 
4\pi^2\p{\al+\gam}\abs{k}^2+4\pi^2\p{\be-\gam}k\otimes k+\del=\p{4\pi^2\p{\al+\be}\abs{k}^2+\del}\p{I-\P}+\p{4\pi^2\p{\al+\gam}\abs{k}^2
+\del}\P,
\end{equation}
from which the result follows.
\end{proof}

The key point of this result is that the decoupled problems \eqref{decoupling_lemma_0} are both damped heat equations.  We thus turn our attention to the solution operators associated to damped heat equations.  We begin by introducing a data flow operator, analogous to the one in Definition \ref{Heat flow 1}.

\begin{defn}\label{Heat flow 2}
Let $\p{r,n}\in\R\times\N$ and $0<T\in\bar{\R}$. For parameters $\mu,\nu\in\R^+$ we define the mapping
\begin{equation} 
\S_{\mu,\nu}:H^{1+r}\p{\T^3;\R^3}\to L^2\p{\p{0,T};H^{2+r}\p{\T^3;\R^3}}\cap H^n\p{\p{0,T};H^{2+r-2n}\p{\T^3;\R^3}}
\end{equation}
for $t \in \R^+$ via
\begin{equation} 
\S_{\mu,\nu}\p{g}\p{t}=\exp\p{-\nu t}\sum_{k\in\Z^3}\bf{e}_k\exp\p{-4\pi^2\mu\abs{k}^2t}\hat{g}\p{k}.
\end{equation}
This defines an $\R^3$-valued distribution by Lemma \ref{real_lemma}.
\end{defn}

We now record the essential properties of the operator $\S_{\mu,\nu}$, including the proof that the codomain of the operator is as stated in the definition.

\begin{prop}\label{heat flow 2}
Let $\p{r,n}\in\R\times\N$, $0<T\in\bar{\R}$, and  $\mu,\nu\in\R^+$.  Consider the operator $\S_{\mu,\nu}$ from Definition \ref{Heat flow 2}.  Then the following hold.
\begin{enumerate}
\item$\S_{\mu,\nu}$ is well-defined, linear, and bounded.
\item If $s\in\R$ is such that $r\le s$, then there exists a constant $C >0$, depending on $\mu,\nu,r,s$, such that  
\begin{equation} 
\norm{\S_{\mu,\nu}\p{g}\p{t}}_{H^{1+s}}\le C\f{\exp\p{-\nu t}}{t^{\f{s-r}{2}}}\norm{g}_{H^{1+r}}
\end{equation}
for every $g \in H^{1+r}(\T^3;\R^3)$ and $t \in \R^+$.  In particular, $\S_{\mu,\nu}$ is smoothing in the sense that if  $g\in H^r\p{\T^3;\R^3}$, then $\S_{\mu,\nu}\p{g}(t) \in C^\infty(\T^3;\R^3)$ for every $t \in \R^+$.

\item  If $g \in H^r_\perp\p{\T^3;\R^3}$, then for each $t\in\R^+$ we have $\mathcal{S}_{\mu,\nu}\p{g}\p{t} \in H^r_\perp\p{\T^3;\R^3}$.  Similarly, if $g\in H^r_\|\p{\T^3;\R^3}$, then for each $t\in\R^+$ we have $\mathcal{S}_{\mu,\nu}\p{g}\p{t} \in H^r_\|\p{\T^3;\R^3}$.

\item  For any $g\in H^r\p{\T^3;\R^3}$ we have 
\begin{equation}
\lim_{t\to 0^+}\norm{\S_{\mu,\nu}\p{g}\p{t}-g}_{H^r}=0, 
\end{equation}
and 
\begin{equation}
 \partial_t \S_{\mu,\nu}\p{g}(t) = \mu \Delta \S_{\mu,\nu}\p{g}(t) - \nu \S_{\mu,\nu}\p{g}(t) \text{ for }t \in \R^+.
\end{equation}

\end{enumerate}
\end{prop}
\begin{proof}
The proof follows from minor modification of the one given for Proposition \ref{heat flow 1}.  We omit the details for the sake of brevity.
\end{proof}

Next we define the related convolution operator, analogous to the one given in Definition \ref{Convolution}.

\begin{defn}\label{Convolution 2}
For $\p{r,n}\in\R\times\N$, $0 < T\in \bar{\R}$, and $\mu,\nu \in \R^+$ we define
\begin{multline} 
\mathcal{S}_{\mu,\nu}\ast: L^2\p{\p{0,T};H^r\p{\T^3;\R^3}}\cap H^{n}\p{\p{0,T};H^{r-2n}\p{\T^3;\R^3}} \\
\to L^2\p{\p{0,T};H^{2+r}\p{\T^3;\R^3}}\cap H^{n+1}\p{\p{0,T};H^{r-2n}\p{\T^3;\R^3}}
\end{multline}
via 
\begin{equation} 
\S_{\mu,\nu}\ast f\p{t}=\sum_{k\in\Z^3}\bf{e}_k\int_{(0,t)}\exp\p{-\nu\p{t-\tau}}\exp\p{-4\pi^2\mu|k|^2(t-\tau)}\hat{f}(\tau,k)\m{d}\tau
\end{equation}
for $t \in (0,T)$.  Again, this defines an $\R^3$-valued distribution by Lemma \ref{real_lemma}.
\end{defn}

The next result shows that the stated codomain for $\S_{\mu,\nu}$ is valid and establishes some other essential properties of this map.

\begin{prop}\label{Convolution 2 well defined}
Let $\p{r,n}\in\R\times\N$, $0 < T\in \bar{\R}$, $\mu,\nu \in \R^+$, and consider the map $\S_{\mu,\nu}\ast$ given in Definition \ref{Convolution 2}.  Then $\S_{\mu,\nu}\ast$ is well-defined with the stated codomain, and it defines a bounded linear map.  Moreover, for any $f \in L^2\p{\p{0,T};H^r\p{\T^3;\R^3}}\cap H^{n}\p{\p{0,T};H^{r-2n}\p{\T^3;\R^3}}$ we have that 
\begin{equation}
 \lim_{t\to 0^+} \norm{\S_{\mu,\nu}\ast f(t)}_{H^{1+r}} =0,
\end{equation}
and 
\begin{equation}
 \partial_t \S_{\mu,\nu}\ast f(t) = \mu \Delta \S_{\mu,\nu}\ast f (t) - \nu \S_{\mu,\nu}\ast f(t) + f(t) \text{ for }t \in (0,T).
\end{equation}

\end{prop}
\begin{proof}
The proof similar to that of Proposition \ref{Convolution well defined}, and so we omit it.
\end{proof}

We now employ the operators $\S_{\mu,\nu}$ and $\S_{\mu,\nu}\ast$ to build an isomorphism associated to the damped heat equation.

\begin{thm}\label{Type 2 iso}
Let $\p{r,n}\in\R\times\N$, $0<T\in\bar{\R}$, and $\mu,\nu \in \R^+$.  Define the mapping
\begin{multline} 
\chi_{\mu,\nu}:L^2\p{\p{0,T};H^{2+r}\p{\T^3;\R^3}} \cap H^{n+1}\p{\p{0,T};H^{r-2n}\p{\T^3;\R^3}} \\
\to H^{1+r}\p{\T^3;\R^3} \times \left( L^2\p{\p{0,T};H^r\p{\T^3;\R^3}}\cap H^n\p{\p{0,T};H^{r-2n}\p{\T^3;\R^3}} \right)
\end{multline}
via
\begin{equation} 
\chi_{\mu,\nu}(u)=\p{u(0),\partial_t u-\mu\Delta u+\nu u},
\end{equation}
where $u(0)$ is understood in the $H^{1+r}$ topology by way of Proposition \ref{interpolation}.  Then $\chi_{\mu,\nu}$ is well-defined and is a bounded linear isomorphism.  Moreover, the following hold.
\begin{enumerate}
\item We have the explicit inverse formula
\begin{equation}\label{Type 2 iso inverse} 
\chi^{-1}_{\mu,\nu}\p{g,f} = \S_{\mu,\nu}g+\S_{\mu,\nu}\ast f.
\end{equation}
\item If $\Gamma\in\left\{\perp,\|\right\}$, $g \in H^{1+r}_\Gamma\p{\T^3;\R^3}$, and $f \in L^2\p{\p{0,T}; H^r_\Gamma\p{\T^3;\R^3}} \cap H^n\p{\p{0,T};H^{r-2n}_\Gamma\p{\T^3;\R^3}}$, then $\chi_{\mu,\nu}^{-1}\p{g,f} \in L^2\p{\p{0,T};H^{2+r}_\Gamma\p{\T^3;\R^3} }\cap H^{n+1}\p{\p{0,T};H^{r-2n}_\Gamma\p{\T^3;\R^3}}.$
\end{enumerate}
\end{thm}
\begin{proof}
As in the proof of Theorem \ref{type 1 solution}, we use Proposition \ref{interpolation} to deduce that $\chi_{\mu,\nu}$ is well-defined and gives a bounded linear map.  Then Propositions \ref{heat flow 2} and \ref{Convolution 2 well defined} show that $\chi_{\mu,\nu}$ is an isomorphism with inverse given by \eqref{Type 2 iso inverse}.  The second stated item follows from the fact that the Leray projector commutes with all of the operators used to define $\chi_{\mu,\nu}$. 
\end{proof}

With Theorem \ref{Type 2 iso} in hand, we can construct an isomorphism associated to \eqref{pde_2}.

\begin{thm}\label{Type 2}
Let $\p{r,n}\in\R\times\N$, $0<T\in\bar{\R}$, $\al,\be,\del\in\R^+$, and $\gam\in\R^+\cup\cb{0}$. Define the mapping
\begin{multline} 
\Xi_{\p{\al,\be,\gam,\del}} : L^2\p{\p{0,T};H^{2+r}\p{\T^3;\R^3}} \cap H^{n+1}\p{\p{0,T};H^{r-2n}\p{\T^3;\R^3}}\\
\to H^{1+r}\p{\T^3;\R^3} \times L^2\p{\p{0,T}H^r\p{\T^3;\R^3}}\cap H^n\p{\p{0,T};H^{r-2n}\p{\T^3;\R^3}}
\end{multline}
via 
\begin{equation} 
\Xi_{\p{\al,\be,\gam,\del}}(u)=\p{u(0),\partial_t u-\p{\al+\gam}\Delta u-\p{\be-\gam}\grad\Div u+\del u},
\end{equation}
where $u(0)$ is understood in the sense of Proposition \ref{interpolation}.  Then  $\Xi_{\al,\be,\gam,\del}$ is well-defined and is a bounded linear isomorphism.
\end{thm}
\begin{proof}
The Leray projector (see Definition \ref{Leray}) induces the following natural isomorphisms:
\begin{multline}
 \mathcal{I}:H^{1+r}\p{\T^3;\R^3} \times \p{ L^2\p{\p{0,T};H^{r}\p{\T^3;\R^3}}\cap H^{n}\p{\p{0,T};H^{r-2n}\p{\T^3;\R^3}} }  \\
 \to \p{H^{1+r}_\perp\p{\T^3;\R^3} \times  L^2\p{\p{0,T};H_\perp^{r}\p{\T^3;\R^3}}\cap H^{n}\p{\p{0,T};H_\perp^{r-2n}\p{\T^3;\R^3}}}\\\oplus \p{H^{1+r}_\|\p{\T^3;\R^3} \times L^2\p{\p{0,T};H^{r}_\|\p{\T^3;\R^3}}\cap H^{n}\p{\p{0,T};H^{r-2n}_\|\p{\T^3;\R^3}}} 
\end{multline}
and 
\begin{multline}
 \mathcal{J}: L^2\p{\p{0,T};H^{2+r}\p{\T^3;\R^3}}\cap H^{n+1}\p{\p{0,T};H^{r-2n}\p{\T^3;\R^3}}  
 \\\to \p{L^2\p{\p{0,T};H_\perp^{2+r}\p{\T^3;\R^3}}\cap H^{n+1}\p{\p{0,T};H_\perp^{r-2n}\p{\T^3;\R^3}}}\\\oplus\p{L^2\p{\p{0,T};H_\|^{2+r}\p{\T^3;\R^3}}\cap H^{n+1}\p{\p{0,T};H_\|^{r-2n}\p{\T^3;\R^3}}}.
\end{multline}
It is then a simple matter to verify that 
\begin{equation} 
\Xi_{(\al,\be,\gam,\del)}=\mathcal{J}^{-1}\circ\p{\chi_{\al+\gam,\del},\chi_{\al+\be,\del}}\circ\mathcal{I},
\end{equation}
and hence $\Xi_{\p{\al,\be,\gam,\delta}}$ is an isomorphism by virtue of Theorem \ref{Type 2 iso}.
\end{proof}

Theorem \ref{Type 2} provides us with a solution operator to \eqref{pde_2}.  It will be useful later to decompose this operator into two parts: one that flows initial data, and one that acts as a `convolution' on forcing.

\begin{defn}\label{Type 2 flow and conv}
Let $\p{r,n}\in\R\times\N$, $0<T\in\bar{\R}$, $\al,\be,\del\in\R^+$, and $\gam\in\R^+\cup\cb{0}$.  We define 
\begin{equation}
 \mathcal{T}_{\al,\be,\gam,\del}:H^{1+r}\p{\T^3;\R^3}\to L^2\p{\p{0,T};H^{2+r}\p{\T^3;\R^3}}\cap H^{n+1}\p{\p{0,T};H^{r-2n}\p{\T^3;\R^3}}
\end{equation}
and 
\begin{multline}
 \mathcal{T}_{\al,\be,\gam,\del}\ast: L^2\p{\p{0,T};H^r\p{\T^3;\R^3}}\cap H^n\p{\p{0,T};H^{r-2n}\p{\T^3;\R^3}} \\
 \to L^2\p{\p{0,T};H^{2+r}\p{\T^3;\R^3}}\cap H^{n+1}\p{\p{0,T};H^{r-2n}\p{\T^3;\R^3}}
\end{multline}
via 
\begin{equation}
 \mathcal{T}_{\al,\be,\gam,\del}(v)=\Xi_{\p{\al,\be,\gam,\del}}^{-1}\p{v,0} \text{ and } \mathcal{T}_{\al,\be,\gam,\del}\ast w=\Xi_{\p{\al,\be,\gam,\del}}^{-1}\p{0,w},
\end{equation}
where $\Xi_{\p{\al,\be,\gam,\del}}$ is as in Theorem \ref{Type 2}.  These are well-defined and bounded linear mappings according to Theorem \ref{Type 2}. 
\end{defn}

To conclude this subsection we state two results about these operators.  The first is analogous to Proposition \ref{1 Preservation}.

\begin{prop}\label{2 Preservation}
Let $\p{r,n}\in\R\times\N$, $0<T\in\bar{\R}$, $\al,\be,\del\in\R^+$, and $\gam\in\R^+\cup\cb{0}$. If $\Gamma\in\cb{\perp,\|}$ and $v \in H^{1+r}_\Gamma\p{\T^3;\R^3}$, and $w \in L^2\p{\p{0,T};H^r_{\Gamma}\p{\T^3;\R^3}} \cap H^n\p{\p{0,T};H^{r-2n}_\Gamma\p{\T^3;\R^3}}$, then 
\begin{equation} 
\mathcal{T}_{\al,\be\gam,\del}\p{v}+\mathcal{T}_{\al,\be,\gam,\del}\ast w\in L^2\p{\p{0,T};H^{2+r}_\Gamma\p{\T^3;\R^3}}\cap H^{n+1}\p{\p{0,T};H^{r-2n}_\Gamma\p{\T^3;\R^3}}.
\end{equation}
\end{prop}
\begin{proof}
This follows directly from the structure of $\Xi_{\p{\al,\be,\gam,\del}}$ recorded in the proof of Theorem \ref{Type 2}.
\end{proof}

The second result is analogous to Proposition  \ref{heat flow 2}.

\begin{prop}\label{type 2 flow smoothing}
Let $\p{r,n}\in\R\times\N$,  $\al,\be,\del\in\R^+$, $\gam\in\R^+\cup\cb{0}$, and $s \in \R$ with $r \le s$.  Then there exists a constant $C >0$, depending on $\alpha$, $\beta$, $\gamma$, $\delta$, $r$, and $s$, such that 
\begin{equation} 
\norm{\mathcal{T}_{\al,\be,\gam,\del}(f)(t)}_{H^s} \le C \f{\exp\left(-\delta t\right)}{t^{\f{s-r}{2}}}\norm{f}_{H^{r}}
\end{equation}
for every  $f\in H^r\p{\T^3;\R^3}$ and $t \in \R^+$.  In particular,  $\mathcal{T}_{\al,\be,\gam,\del}$ is smoothing.
\end{prop}
\begin{proof}
This follows from the structure of $\Xi_{\p{\al,\be,\gam,\del}}$  recorded in the proof of Theorem \ref{Type 2}, coupled with the second item  of Proposition \ref{heat flow 2}.
\end{proof}

\subsection{Potential microflow solutions to \eqref{micro_polar}}

We now have all of the tools needed to rigorously construct the micropotential solutions to \eqref{micro_polar} with $u=f=g=0$.

\begin{defn}\label{potential flow}
Given a regularity parameter $q\in\R$ and a conservative field $\zeta_0\in H^q_\|\p{\T^3;\R^3}$, we define the potential microflow solution generated by $\zeta_0$ via 
\begin{equation}
\zeta=\mathcal{T}_{\f{\al}{j},\f{\al+3\be}{3j},\f{\gam}{j},\f{2\kappa}{j}}\p{\zeta_0}, 
\end{equation}
where $\mathcal{T}_{\f{\al}{j},\f{\al+3\be}{3j},\f{\gam}{j},\f{2\kappa}{j}}$ is as in Definition \ref{Type 2 flow and conv}. 
\end{defn}

The following summarizes the essential facts about these special solutions.

\begin{thm}\label{micropotential_exist}
Let $q \in \R$, $\zeta_0\in H^q_\|\p{\T^3;\R^3}$, and let $\zeta$ be as in Definition \ref{potential flow}.  Then the following hold.
\begin{enumerate}
 \item For each $n \in \N$ we have the inclusion $\zeta\in L^2\p{\R^+;H^{1+q}_\|\p{\T^3;\R^3}}\cap H^{n+1}\p{\R^+;H_\|^{q-2n-1}\p{\T^3;\R^3}}$, and there exists a constant $C>0$, independent of $\zeta_0$, such that 
\begin{equation}
 \norm{\zeta}_{L^2 H^{1+q}_\|} + \norm{\zeta}_{H^{n+1}H_\|^{q-2n-1} } \le C \norm{\zeta_0}_{H^q_\|}.
\end{equation}

 \item $\zeta$ is smooth on the space-time domain, i.e. $\zeta\in C^\infty\p{\R^+\times\T^3;\R^3}$
 \item $\zeta$ solves the initial-value problem
\begin{equation} 
\begin{cases}
j\pd_t\zeta-\p{\al+\gam}\Delta\zeta-\p{\f{\al}{3}+\be-\gam}\grad\Div\zeta+2\kappa\zeta=0 &\text{in }\R^+\times\T^3\\
\curl \zeta =0 &\text{in }\R^+\times\T^3\\
\zeta\p{0}=\zeta_0&\text{on } \T^3,
\end{cases}
\end{equation}
where the initial condition is understood in the $H^q_\|\p{\T^3;\R^3}$ topology via Proposition \ref{interpolation}. In particular, the triple $u=0$, $p=0$, $\omega = \zeta$ solves \eqref{micro_polar} with $f=g=0.$
\end{enumerate}
\end{thm}
\begin{proof}
The first and third items follow from the definition of $\mathcal{T}_{\f{\al}{j},\f{\al+3\be}{3j},\f{\gam}{j},\f{2\kappa}{j}}$ and Proposition \ref{2 Preservation}.  The second item follows from Proposition \ref{type 2 flow smoothing} and the first item. 
\end{proof}

\section{Linearization of \eqref{micro_polar} around a potential microflow}\label{sec_linearization}

In this section we consider a linearization of \eqref{micro_polar}  around a potential microflow from Definition \ref{potential flow}.  More precisely, we assume that $\zeta_0\in H^q_\|\p{\T^3;\R^3}$ is given for some $q \in \R^+$ and we consider the potential microflow $\zeta\in L^2\p{\R^+;H^{1+q}_\|\p{\T^3;\R^3}}$  as in Definition \ref{potential flow} (see also Theorem \ref{micropotential_exist} for the properties of $\zeta$).  We then consider the problem of finding $u,\omega : \R^+ \times \T^3 \to \R^3$ and $p:\R^+\times\T^3\to\R$ solving the linear initial-value problem
\begin{equation}\label{mp_linearized_} 
\begin{cases}
\Div u=0 & \text{in } \R^+\times\T^3\\
\varrho \pd_t u-\p{\ep+\frac{\kappa}{2} }\Delta u- \kappa \curl\omega + \grad p=f & \text{in } \R^+\times\T^3\\
j \pd_t\omega + j u\cdot \grad\zeta-\p{\al+\gam}\Delta\omega-\p{\frac{\al}{3}+\be-\gam} \grad\Div\omega + 2\kappa \omega -  \kappa\curl u=\f{1}{j}g & \text{in } \R^+\times\T^3\\
\p{u\p{0},\omega\p{0}}=\p{u_0,\omega_0}&\text{on } \T^3
\end{cases}
\end{equation}
for given data $u_0,\omega_0 : \T^3 \to \R^3$ and $f,g : \R^+ \times \T^3  \to \R^3$ in certain regularity classes, with $f$ having vanishing spatial average for all time and $u_0$ having vanishing spatial average and divergence.

The evolution of the pressure is essentially trivial since the forcing term in known. If we apply the operator $\p{I-\P}$ to the second equation in \eqref{mp_linearized_} we are left with: $\grad p=\p{I-\P}f$. As $\int_{\T^3}p\p{t,x}\;\m{d}x=0$ for almost every $t\in\R^+$, we are in a position to use Proposition \ref{potential_map} to deduce that  $p=\Pi\p{I-\P}f$. To streamline the linear existence theory, we will posit in addition that $\p{I-\P}f=0$ on the space-time domain.  Thus, we shall study the system:
\begin{equation}\label{mp_linearized} 
\begin{cases}
\Div u=0 & \text{in } \R^+\times\T^3\\
\varrho \pd_t u-\p{\ep+\frac{\kappa}{2} }\Delta u- \kappa \curl\omega =f & \text{in } \R^+\times\T^3\\
j \pd_t\omega + j u\cdot \grad\zeta-\p{\al+\gam}\Delta\omega-\p{\frac{\al}{3}+\be-\gam} \grad\Div\omega + 2\kappa \omega -  \kappa\curl u= g & \text{in } \R^+\times\T^3\\
\p{u\p{0},\omega\p{0}}=\p{u_0,\omega_0}&\text{on } \T^3.
\end{cases}
\end{equation}
We will see later that the well-posedness of this reduced linear system is sufficient for our intended purpose.

Since \eqref{mp_linearized} is a non-constant coefficient problem, the Fourier transform is not a particularly convenient technique for producing a solution.  Instead, our strategy for solving \eqref{mp_linearized} is to employ a fixed-point argument.  For this it's convenient to initially work in a functional setting with minimal temporal regularity and integrability and to subsequently bootstrap.

\subsection{Locally integrable solutions and bootstrapping}

We now define a notion of solution to \eqref{mp_linearized} that we call a locally integrable solution.  It has all of the desired spatial regularity but lacks high-order temporal regularity and integrability.  

\begin{defn}\label{locally integrable solution}
Let $s\in\R^+\cup\cb{0}$ and 
\begin{equation}
 q \in 
\begin{cases}
(3/2,\infty) &\text{if } 0 \le s \le 1/2 \\
(s,\infty)   &\text{if } 1/2 < s.
\end{cases}
\end{equation}
Let $\zeta_0\in H^q_\|\p{\T^3;\R^3}$, and let $\zeta$ be the associated potential microflow from Definition \ref{potential flow}.  Suppose that $u_0 \in \z{H}^{1+s}_{\perp}\p{\T^3;\R^3}$, $\omega_0 \in H^{1+s}\p{\T^3;\R^3}$, $f \in L^2\p{\R^+;\z{H}^s_\perp\p{\T^3;\R^3}}$, and $g \in  L^2\p{\R^+;H^s\p{\T^3;\R^3}}$.  A locally integrable solution to \eqref{mp_linearized} is a pair $u,\omega : \R^+ \times \T^3 \to \R^3$ such that
\begin{equation}
u \in \bigcap_{T\in\R^+}L^2\p{\p{0,T};\z{H}^{2+s}_\perp\p{\T^3;\R^3}}\cap H^1\p{\p{0,T};\z{H}_\perp^{s}\p{\T^3;\R^3}} 
\end{equation}
and 
\begin{equation} 
\omega \in \bigcap_{T\in\R^+}L^2\p{\p{0,T};H^{2+s}\p{\T^3;\R^3}}\cap H^1\p{\p{0,T};H^s\p{\T^3;\R^3}},
\end{equation}
satisfying \eqref{mp_linearized} in the strong sense, with the initial data $(u_0,\omega_0)$ achieved in the $H^{1+s}$ topology as in Proposition \ref{interpolation}.  
\end{defn}

Locally integrable solutions have minimal temporal regularity and integrability, but we will show that the structure of \eqref{mp_linearized} automatically promotes solutions to higher temporal integrability.  To prove this we begin by introducing some useful functionals

\begin{defn}\label{functionals}
Suppose that $s,$ $q$, $\zeta_0$, $\zeta$, $f$, $g$, $u_0$, $\omega_0$, $u$, and $\omega$ are as in Definition \ref{locally integrable solution}.  We define $\mathscr{E},\mathscr{D},\mathscr{F} : \R^+ \to \R$ via 
\begin{equation}
\mathscr{E}\p{t}=\int_{\T^3}\f{\varrho}{2}\abs{\J^su\p{t}}^2+\f{j}{2}\abs{\J^s\omega\p{t}}^2, 
\end{equation}
\begin{multline}
\mathscr{D}\p{t}=\int_{\T^3} \ep\abs{D\J^s u\p{t}}^2+\p{\al+\gam}\abs{D\J^s\omega\p{t}}^2 \\
+\int_{\T^3} \p{\f{\al}{3}+\be-\gam}\abs{\Div\J^s\omega\p{t}}^2+2\kappa\abs{\f12\curl\J^su\p{t}-\J^s\omega\p{t}}^2, 
\end{multline}
and 
\begin{equation}
\mathscr{F}\p{t}= \int_{\T^3}\J^sf\p{t} \cdot \J^s u\p{t} + \J^sg\p{t}\cdot \J^s\omega\p{t} - \J^s\p{u\p{t} \cdot \grad\zeta\p{t}} \cdot \J^s\omega\p{t},
\end{equation}
    where $\J^s$ is the operator given in Definition \ref{psuedo diff}.  Note that the properties of the locally integrable solution show that $\mathscr{E}\in \bigcap_{T\in\R^+}H^1\p{0,T}$ and $\mathscr{D},\mathscr{F} \in \bigcap_{T\in\R^+}L^2\p{0,T}$.  We shall call $\mathscr{E}$ the energy functional, $\mathscr{D}$ the dissipative functional, and $\mathscr{F}$ the forcing functional.
\end{defn}

Our next result shows how $\mathscr{E}$, $\mathscr{D}$, and $\mathscr{F}$ are related to one another.

\begin{lem}\label{ED}
The functionals $\mathscr{E}$, $\mathscr{D}$, and $\mathscr{F}$ from Definition \ref{functionals} are related by the identity 
\begin{equation}
\mathscr{E}'+\mathscr{D}=\mathscr{F} \text{ almost everywhere in } \R^+.
\end{equation}
\end{lem}
\begin{proof}
By hypothesis, a locally integrable solution pair $(u,\omega)$ is a strong solution to \eqref{mp_linearized}.  We apply the operator $\J^s$ to the second and third equations in \eqref{mp_linearized}, take the $L^2$ inner-product with $\J^s u\p{t}$ and $\J^s\omega\p{t}$ in $H^s\p{\T^3;\R^3}$, respectively, and then sum.   This results in the following identity (omitting the $t$-dependence for brevity):
\begin{multline}
\int_{\T^3}\varrho\p{\J^s u}'\cdot \J^s u + j\p{\J^s\omega}'\cdot \J^s\omega -\p{\ep+\f{\kappa}{2}}\Delta\J^s u \cdot \J^s u-\p{\al+\gam}\Delta\J^s\omega \cdot \J^s\omega-\p{\f{\al}{3}+\be-\gam}\grad\Div\J^s\omega \cdot \J^s\omega\\
+ \int_{\T^3} 2\kappa \J^s\omega \cdot \J^s\omega - \kappa \curl\J^s \omega \cdot \J^su - \kappa\curl \J^su \cdot \J^s\omega = \int_{\T^3}\J^sf\cdot \J^s u+\J^s g \cdot \J^s\omega - \J^s \p{u\cdot \grad\zeta}\cdot \J^s\omega.
\end{multline}
Since $\Div \J^s u =0$, the square norm of the $\nabla \J^s u$ is equal to the square norm of $\curl \J^s u$. Consequentially, we can recognize the  perfect square
\begin{equation} 
\int_{\T^3}\f{\kappa}{2}\Delta\J^su\cdot \J^s u - \kappa \curl \J^s\omega \cdot \J^su - \kappa \curl\J^su \cdot \J^s\omega + 2\kappa\J^s\omega\cdot \J^s\omega = 2\kappa\int_{\T^3}\abs{\f{1}{2}\curl\J^su-\J^s\omega}^2.
\end{equation}
On the other hand, integration by parts shows that 
\begin{multline}
\int_{\T^3}  -\ep \Delta\J^s u\cdot \J^s u-\p{\al+\gam}\Delta\J^s\omega\cdot \J^s\omega-\p{\f{\al}{3}+\be-\gam}\grad\Div\J^s\omega\cdot \J^s\omega  \\
= \int_{\T^3}\ep\abs{D\J^s u }^2 + \p{\al+\gam}\abs{D\J^s\omega}^2 
+ \p{\f{\al}{3}+\be-\gam}\abs{\Div\J^s\omega}^2. 
\end{multline}
The result now follows by 
noting that for almost every time $t \in \R^+$
\begin{equation}
\mathscr{E}'= \int_{\T^3}\varrho\p{\J^s u}'\cdot \J^s u + j\p{\J^s\omega}'\cdot \J^s\omega. 
\end{equation}
\end{proof}

Next we prove that the dissipation functional is coercive over the energy functional.

\begin{lem}\label{ED estimates}
Let the functionals $\mathscr{E}$ and $\mathscr{D}$ be as defined in Definition \ref{locally integrable solution}.  Let 
\begin{equation}
\mathcal{C}_0=\min\cb{\f{\pi^2\ep}{\varrho},\f{\ep\kappa}{2j\p{\ep+\kappa}}} \text{ and } \mathcal{C}_1=\min\cb{\f{\pi^2\ep}{2},\f{\ep\kappa}{4\p{\ep+\kappa}},2\al,3\be,2\gam}. 
\end{equation}
Then for a.e. $t\in\R^{+}$ we have that
\begin{multline} 
\mathscr{D}\p{t}\ge\mathcal{C}_0\p{\f{\varrho}{2}\norm{u\p{t}}_{\z{H}^s_\perp}^2+\f{j}{2}\norm{\omega\p{t}}_{H^s}^2}+\mathcal{C}_1\p{\norm{u\p{t}}_{\z{H}^{1+s}_{\perp}}^2+\norm{\omega\p{t}}_{H^{1+s}}^2} \\
=\mathcal{C}_0\mathcal{E}\p{t}+\mathcal{C}_1\norm{\p{u\p{t},\omega\p{t}}}_{\z{H}^{1+s}_\perp\times H^{1+s}}^2.
\end{multline}
In particular, we have that $\mathscr{D} \ge 0$.
\end{lem}

\begin{proof}
We again suppress the time dependence for brevity.  We compute that 
\begin{equation}
\Div\p{\al\p{D\omega+D\omega^{\m{t}}-\f23\Div\omega I}+\be\Div\omega I+\gam\p{D\omega+D\omega^\m{t}}}=\p{\al+\gam}\Delta\omega+\p{\f{\al}{3}+\be-\gam}\grad\Div\omega. 
\end{equation}
Taking the inner-product of this equation with $\omega$ in the space $H^s\p{\T^3;\R^3}$ and integrating by parts yields the identity
\begin{multline}
\int_{\T^3}\Div\p{\al\p{D\J^s\omega+D\J^s\omega^{\m{t}}}+\be\Div\J^s\omega I +\gam\p{D\J^s\omega+D\J^s\omega^\m{t}}}\cdot\J^s\omega \\
=-\int_{\T^3}\p{\al+\gam}\abs{D\J^s\omega}^2+\p{\f{\al}{3}+\be-\gam}\abs{\Div\J^s\omega}^2.
\end{multline}
We write 
\begin{equation}
D\omega=\f12\p{D\omega+D\omega^{\m{t}}-\f23\Div\omega I}+\f13\Div\omega I+\f12\p{D\omega-D\omega^{\m{t}}}, 
\end{equation}
which is an orthogonal decomposition relative to the usual Frobenius inner-product on matrices.  Thus the previous identity yields
\begin{multline}
\int_{\T^3}\p{\al+\gam}\abs{D\J^s\omega}^2+\p{\f{\al}{3} + \be-\gam}\abs{\Div\J^s\omega}^2  \\
= \int_{\Omega}2\al\abs{\f{1}{2}\p{D\J^s\omega+D\J^s\omega^{\m{t}}-\f23\Div\J^s\omega I}}^2  + 3\be\abs{\f13\Div\J^s\omega I}^2+2\gam\abs{\f{1}{2}\p{D\J^s\omega-D\J^s\omega^{\m{t}}}}^2 \\
\ge \min\cb{2\al,3\be,2\gam}\int_{\T^3}\abs{D\J^s\omega}^2. 
\end{multline}

Let $\del=\f{1}{4}\p{1+\f{\ep}{\kappa}}\in\R^+$.  We may use Cauchy's inequality and the fact that $u$ is solenoidal to bound
\begin{multline}
\int_{\T^3}\ep\abs{D\J^su}^2+2\kappa\abs{\f12\curl\J^su-\J^s\omega}^2=\int_{\T^3}\p{\ep+\f{\kappa}{2}}\abs{D\J^su}^2+2\kappa\abs{\J^s\omega}^2-2\kappa\curl\J^su\cdot \J^s\omega \\
 \ge \int_{\T^3}\p{\ep+\f{\kappa}{2}-2\kappa\del}\abs{D\J^su}^2+\p{2\kappa-\f{\kappa}{2\del}}\abs{\J^s\omega}^2=\int_{\T^3}\f{\ep}{2}\abs{D\J^su}^2+\f{\ep\kappa}{2\p{\ep+\kappa}}\abs{\J^s\omega}^2.
\end{multline}
Since $u$ has zero spatial average,  $\J^s u$ does as well, and so we have the Poincar\'e inequality
\begin{equation}
\int_{\T^3}\abs{D\J^su}^2\ge4\pi^2\int_{\T^3}\abs{\J^su}^2. 
\end{equation}
Putting  the above estimates together, we arrive at the bound
\begin{multline}
\mathscr{D}\p{t}  \ge \min\cb{\f{\pi^2\ep}{\varrho},\f{\ep\kappa}{2j\p{\ep+\kappa}}}\p{\int_{\T^3}\f{\varrho}{2}\abs{\J^su\p{t}}^2+\f{j}{2}\abs{\J^s\omega\p{t}}^2} \\
+\min\cb{\f{\pi^2\ep}{2},\f{\ep\kappa}{4\p{\ep+\kappa}},2\al,3\be,2\gam}\p{\int_{\T^3}\abs{\J^{s+1}u\p{t}}^2+\abs{\J^{s+1}\omega\p{t}}^2},
\end{multline}
which is the desired result.

\end{proof}

We now have the tools needed to begin bootstrapping.  We will do so in two steps.

\begin{thm}[First promotion of locally integrable solutions]\label{promotion 1}
Suppose that $s,$ $q$, $\zeta_0$, $\zeta$, $f$, $g$, $u_0$, $\omega_0$, $u$, and $\omega$ are as in Definition \ref{locally integrable solution}.   Then we have the inclusions
\begin{equation} 
u \in  L^\infty\p{\R^+;\z{H}^s_\perp \p{\T^3;\R^3}}\cap L^2\p{\R^+;\z{H}^{1+s}_\perp\p{\T^3;\R^3}}
\end{equation}
and 
\begin{equation} 
\omega \in  L^\infty\p{\R^+;H^s\p{\T^3;\R^3}}\cap L^2\p{\R^+;H^{1+s}\p{\T^3;\R^3}}.
\end{equation}
Moreover, there exists a constant $\mathcal{C}\in\R^+$, independent of $u$, $\omega$, $f$, $g$, $u_0$, and $\omega_0$, such that
\begin{equation} 
\norm{u}_{L^\infty\z{H}^s_\perp}+\norm{\omega}_{L^\infty H^s}+\norm{u}_{L^2\z{H}^{1+s}_\perp}+\norm{\omega}_{L^2H^{1+s}}\le\mathcal{C}\p{\norm{u_0}_{\z{H}^{1+s}_\perp}+\norm{\omega_0}_{H^{1+s}}+\norm{f}_{L^2\z{H}^s_\perp}+\norm{g}_{L^2H^s}}.
\end{equation}
\end{thm}
\begin{proof}
Due to the dependence of $q$ on $s$ in Definition \ref{locally integrable solution} we must break to two cases:
$0 \le s \le 1/2$ and $1/2 < s$.   We begin with the harder case, $0 \le s \le 1/2$, in which case $\zeta_0\in H^q_\|\p{\T^3;\R^3}$ for $q\in\p{\f32,\infty}$.

Lemmas \ref{ED} and $\ref{ED estimates}$ imply the differential inequality
\begin{equation}\label{promotion 1 1}
\mathscr{E}'\p{t}+\mathcal{C}_0\mathcal{E}\p{t}+\mathcal{C}_1\norm{\p{u\p{t},\omega\p{t}}}^2_{\z{H}^{1+s}_\perp\times H^{1+s}}\le\mathscr{F}\p{t} 
\end{equation}
for a.e. $t\in\R^+$. 

We now turn to an estimate of the functional $\mathscr{F}$.  Picking $\tilde q\in\p{\f32,q}$ and using Propositions \ref{products} and \ref{type 2 flow smoothing}, we learn that there are universal constants $C_0,C_1\in\R^+$ such that  
\begin{multline} 
\norm{u\p{t}\cdot \grad\zeta\p{t}}^2_{H^s}\le C_0\norm{u\p{t}}^2_{H^s}\norm{\grad\zeta\p{t}}_{H^{\tilde{q}}}^2\le 4\pi^2C_0\norm{u\p{t}}^2_{H^s}\norm{\zeta\p{t}}_{H^{\tilde{q}+1}}^2 \\
\le 4\pi^2C_0C_1\f{\exp\p{-\f{4\kappa}{j}t}}{t^{\min\cb{0,1+\tilde{q}-q}}}\norm{u\p{t}}^2_{H^s}\norm{\zeta_0}_{H^{q}}^2.
\end{multline}
From this, the definition of $\mathscr{E}$, and Cauchy's inequality, we obtain the bound
\begin{multline}
\int_{\T^3}\J^s\p{u\p{t}\cdot \grad\zeta\p{t}}\cdot \J^s\omega\p{t} \le 2 \pi \sqrt{C_0C_1}\f{\exp\p{-\f{2\kappa}{j}t}}{t^{\f12\min\cb{0,1+\tilde{q}-q}
}}\norm{\zeta_0}_{H^{q}}\norm{\omega\p{t}}_{H^s}\norm{u\p{t}}_{H^s}\\
 \le 2\pi\sqrt{C_0C_1}\max\cb{\f{1}{\varrho},\f{1}{j}}\f{\exp\p{-\f{2\kappa}{j}t}}{t^{\f12\min\cb{0,1+\tilde{q}-q}}}\norm{\zeta_0}_{H^q}\mathscr{E}\p{t}.
\end{multline}
Set $C_2=2\pi\sqrt{C_0C_1}\max\cb{\f{1}{\varrho},\f{1}{j}}\norm{\zeta_0}_{H^{q}}\in\R^+\cup\cb{0}$. With this and  another use of the Cauchy's inequality, we obtain a good upper bound on the functional $\mathscr{F}$:
\begin{equation}\label{promotion 1 2} 
\mathscr{F}\p{t}\le\p{\f{\mathcal{C}_0}{2}+C_2\f{\exp\p{-\f{2\kappa}{j}t}}{t^{\f12\min\cb{0,1+\tilde q-q}}}}\mathscr{E}\p{t}+\f{1}{\mathcal{C}_0}\max\cb{\f{1}{\varrho},\f{1}{j}}\norm{\p{f\p{t},g\p{t}}}_{\z{H}^s_\perp\times H^s}^2.
\end{equation}

Now set 
\begin{equation}
\mathscr{J}_0\p{t}=\mathcal{C}_1\norm{\p{u\p{t},\omega\p{t}}}_{\z{H}^{1+s}_\perp\times H^{1+s}}^2 \text{ and } \mathscr{J}_1\p{t}=\f{1}{\mathcal{C}_0}\max\cb{\f{1}{\varrho},\f{1}{j}}\norm{\p{f\p{t},g\p{t}}}^2_{\z{H}^{s}_\perp\times H^s}.
\end{equation}
Combining \eqref{promotion 1 1} and \eqref{promotion 1 2} then provides us with the differential inequality
\begin{equation}\label{promotion 1 3} 
\mathscr{E}'\p{t}+\f{\mathcal{C}_0}{4}\mathscr{E}\p{t}+\mathscr{J}_0\p{t}\le\mathscr{J}_1\p{t}+\p{-\f{\mathcal{C}_0}{4}+C_2\f{\exp\p{-\f{2\kappa}{j}t}}{t^{\f12\min\cb{0,1+\tilde q-q}}}}\mathscr{E}\p{t}.
\end{equation}
Now, thanks to monotonicity, there exists a $T^\star\in\R^+\cup\cb{0}$ depending only on $\mathcal{C}_0,C_2,\kappa,j,q,\tilde q$, such that  
\begin{equation}
-\f{\mathcal{C}_0}{4}+C_2\f{\exp\p{-\f{2\kappa}{j}t}}{t^{\f12\min\cb{0,1+\tilde q-q}}}<0 \Leftrightarrow t\in\p{T^\star,\infty}.
\end{equation}
Since $\min\cb{0,1+\tilde{q}-q}<1$, we have that 
\begin{equation}
C_3 = \int_{\p{0,T^\star}}\abs{-\f{\mathcal{C}_0}{4}+C_2\f{\exp\p{-\f{2\kappa}{j}\tau}}{\tau^{\f12\min\cb{0,1+\tilde{q}-q}}}}\;\m{d}\tau < \infty.
\end{equation}
Thus, upon integrating \eqref{promotion 1 3} on the interval $\p{0,t}$ for some $t\in\R^+$,  we arrive at the bound
\begin{equation} 
\mathscr{E}\p{t}-\mathscr{E}\p{0}+\int_{\p{0,t}}\mathscr{J}_0\p{\tau}\;\m{d}\tau\le\int_{\p{0,t}}\mathscr{J}_1\p{\tau}\;\m{d}\tau+C_3\sup_{\tau\in\p{0,T^\star}}\mathscr{E}\p{\tau}.
\end{equation}
We then take the supremum over all $t\in\R^+$ to obtain
\begin{equation}\label{promotion 1 4}
\sup_{t\in\R^+}\mathscr{E}\p{t}+\int_{\R^+}\mathscr{J}_0\le\mathscr{E}\p{0}+\int_{\R^+}\mathscr{J}_1+C_3\sup_{\tau\in\p{0,T^\star}}\mathscr{E}\p{\tau}.
\end{equation}

Next we aim to estimate the right-most term in \eqref{promotion 1 4} in terms of the initial data and forcing. Set $\eta\in L^1\p{(0,T^\star)}$ via 
\begin{equation}
\eta\p{t}=\f{\mathcal{C}_0}{2}-C_2\f{\exp\p{-\f{2\kappa}{j}t}}{t^{\f12\min\cb{0,1+\tilde{q}-q}}}. 
\end{equation}
Returning to $\eqref{promotion 1 3}$,  we find that  a.e. $t\in\p{0,T^\star}$ 
\begin{equation}
\mathscr{E}'\p{t}+\eta\p{t}\mathscr{E}\p{t}\le\mathscr{J}_1\p{t}. 
\end{equation}
Applying Gronwall's lemma and taking the supremum, we arrive at the estimate
\begin{multline} 
\sup_{t \in (0,T^\star)} \mathscr{E}\p{t} \le \sup_{t \in (0,T^\star)} \left[ \exp\p{-\int_{\p{0,t}}\eta\p{\tau}\;\m{d}\tau}\mathscr{E}\p{0}+\int_{\p{0,t}}\exp\p{-\int_{\p{\tau,t}}\eta\p{\mu}\;\m{d}\mu}\mathscr{J}_1\p{\tau}\;\m{d}\tau \right] \\
\le \exp\p{C_3}\p{\mathscr{E}\p{0}+\int_{\R^+}\mathscr{J}_1}.
\end{multline}
Combining this with \eqref{promotion 1 4} proves the theorem in the cases where $s\in\sb{0,\f{1}{2}}$. 

In the second case $s\in\p{\f12,\infty}$, in which case we are assuming that $\zeta_0\in H^q_\|\p{\T^3;\R^3}$ with $q\in\p{s,\infty}$.  The only difference in the argument is that for any $\del\in\R^+$  we can now bound 
\begin{multline} 
\int_{\T^3}\J^s\p{u\p{t} \cdot \grad\zeta\p{t}} \cdot \J^s\omega\p{t} 
\le c\p{ \norm{u\p{t}}_{\z{H}^{1+s}_\perp}\norm{\omega\p{t}}_{H^{s}}\norm{\zeta\p{t}}_{H^{1+s}}} \\
\le c\p{ \del\norm{u\p{t}}_{\z{H}^{1+s}_\perp}^2+\f{1}{4\del}\norm{\omega\p{t}}_{H^{s}}^2\norm{\zeta\p{t}}_{H^{1+s}}^2},
\end{multline}
where $c\in\R^+$ is a constant depending only on $s$ from Proposition \ref{products}. We then choose $\del$ so small that the term with $u$ can be absorbed onto the left within the dissipative functional $\mathscr{D}$, and then we repeat the same argument as above to conclude.
\end{proof}

We can further exploit the structure of \eqref{mp_linearized} to improve the result of the previous theorem.

\begin{thm}[Second promotion of locally integrable solution]\label{promotion 2}
Suppose that $s,$ $q$, $\zeta_0$, $\zeta$, $f$, $g$, $u_0$, $\omega_0$, $u$, and $\omega$ are as in Definition \ref{locally integrable solution}.  Then we have the inclusions
\begin{equation} 
u \in  L^2\p{\R^+;\z{H}^{2+s}_\perp\p{\T^3;\R^3}}\cap H^1\p{\R^+;\z{H}^{s}_\perp\p{\T^3;\R^3}}
\end{equation}
and 
\begin{equation} 
\omega \in L^2\p{\R^+;H^{2+s}\p{\T^3;\R^3}}\cap H^1\p{\R^+;H^{s}\p{\T^3;\R^3}} .
\end{equation}
Moreover, there is a constant $\mathcal{K}\in\R^+$, independent of $u$, $\omega$, $f$, $g$, $u_0$, and $\omega_0$, such that 
\begin{equation} 
\norm{\p{u,\omega}}_{L^2\z{H}^{2+s}_\perp\times L^2H^{2+s}}+\norm{\p{u,\omega}}_{H^1\z{H}^s_\perp\times H^1H^s}\le\mathcal{K}\p{\norm{\p{u_0,\omega_0}}_{\z{H}^{1+s}_\perp\times H^{1+s}}+\norm{\p{f,g}}_{L^2\z{H}^s_\perp\times L^2H^s}}.
\end{equation}
\end{thm}

\begin{proof}
Recall the heat flow and convolution operators of Definitions \ref{Heat flow 1}, \ref{Convolution}, and \ref{Type 2 flow and conv}. For any $T\in\R^+$, we have that $u\in L^2\p{\p{0,T};\z{H}^{2+s}_\perp\p{\T^3;\R^3}}\cap H^1\p{\p{0,T};\z{H}^s_\perp\p{\T^3;\R^3}}$ and that $u$ is a strong solution to the initial value problem
\begin{equation} 
\begin{cases}
\partial_t u  - \f{2\ep+\kappa}{2\varrho}\Delta u =\f{1}{\varrho}f +\f{\kappa}{\varrho}\curl\omega &\text{in } t\in\p{0,T} \times \T^3\\\
u\p{0}=u_0 &\text{on } \T^3.
\end{cases}
\end{equation}
Here we know that $u_0\in\z{H}^{1+s}_\perp\p{\T^3;\R^3}$ and  $\curl\omega,f\in L^2\p{\p{0,T};\z{H}^{s}_\perp\p{\T^3;\R^3}}$ by hypothesis. Thus, Theorems \ref{type 1 solution} and \ref{1 Preservation} imply that 
\begin{equation}
u\p{t}=\mathcal{S}\p{u_0}\p{t}+\mathcal{S}\ast\p{\f{1}{\varrho}f+\f{\kappa}{\varrho}\curl\omega}\p{t} 
\end{equation}
for all $t\in\p{0,T}$, where we have abbreviated $\mathcal{S}=\mathcal{S}_{\f{2\ep+\kappa}{2\varrho}}$. Since $T\in\R^+$ was arbitrary, this identity holds for all $t\in\R^+$.  By Theorem \ref{promotion 1} we have that $\f{1}{\varrho}f+\f{\kappa}{\varrho}\curl\omega\in L^2\p{\R^+;\z{H}^{s}_\perp\p{\T^3;\R^3}}$. Hence, Propositions \ref{heat flow 1} and \ref{Convolution well defined} guarantee  that $u\in L^2\p{\R^+;\z{H}^{2+s}_\perp\p{\T^3;\R^3}}\cap H^1\p{\R^+;\z{H}^s_\perp\p{\T^3;\R^3}}$ and that there is a universal constant $K_1\in\R^+$ such that 
\begin{equation} \label{promotion 2 1}
\norm{u}_{L^2\z{H}^{2+s}_\perp}+\norm{u}_{H^1\z{H}^s_\perp}\le K_1\p{\norm{u_0}_{\z{H}^{1+s}_\perp}+\norm{f}_{L^2\z{H}^s_\perp}+\norm{\curl\omega}_{L^2\z{H}^s_\perp}}.
\end{equation}

To perform similar estimates with $\omega$, we need to estimate the product term $u\cdot \grad\zeta$.  If $s\in\sb{0,\f12}$, then we pick $\tilde{q}\in\p{\f32,q}$ and use Propositions \ref{products} and \ref{type 2 flow smoothing} to bound 
\begin{equation}
\norm{u\cdot \grad\zeta}_{L^2\p{\R^+;\z{H}^s_\perp}}^2\le c\p{\int_{\R^+}\norm{u}_{H^s}^2\norm{\grad\zeta}_{H^{\tilde q}}^2}\le c'\p{\norm{u}_{L^\infty\p{\R^+;\z{H}^s_\perp}}^2\norm{\zeta_0}_{H^q_\|}^2\int_{\R^+}\f{\exp\p{-\f{4\kappa}{j}t}}{t^{\min\cb{0,1+\tilde{q}-q}}}\;\m{d}t}.
\end{equation}
If $s\in\p{\f12,\infty}$ then $s+1 > 3/2$, which allows us to again use  Propositions \ref{products} and \ref{type 2 flow smoothing} to bound 
\begin{equation}
\norm{u\cdot \grad\zeta}^2_{L^2\p{\R^+;H^s}} \le c\p{ \int_{\R^+}\norm{u}^2_{\z{H}^{1+s}_\perp}\norm{\grad\zeta}^2_{H^s}}\le c'\p{\norm{u}_{L^\infty\p{\R^+;\z{H}^{1+s}_\perp}}^2\norm{\zeta_0}^2_{H^q_\|}\int_{\R^+}\f{\exp\p{-\f{4\kappa}{j}t}}{t^{\min\cb{0,s+1-q}}}\;\m{d}t};
\end{equation}
where $c,c'\in\R^+$ are constants depending only on $s$ and $q$. In either case there is a constant $K_2\in\R^+$, independent of $u$, $\omega$, $u_0$, $\omega_0$, $f$, and $g$, such that 
\begin{equation}
\norm{u\cdot \grad\zeta}_{L^2\p{\R^+;H^s}}\le K_2\norm{u}_{L^\infty\p{\R^+;\z{H}^{1+s}_\perp}},
\end{equation}
which is finite by Theorem \ref{promotion 1} and Proposition \ref{interpolation}.  Hence, we can argue exactly as with the velocity field, $u$, to deduce that for all $t\in\R^+$ we have the identity
\begin{equation}
\omega\p{t}=\mathcal{T}\p{\omega_0}\p{t}+\mathcal{T}\ast\p{\f{1}{j}g+\f{\kappa}{j}\curl u-u\cdot \grad\zeta} \p{t}, 
\end{equation}
where we have abbreviated $\mathcal{T}=\mathcal{T}_{\f{\al}{j},\f{\al+3\be}{3j},\f{\gamma}{3},\f{2\kappa}{j}}$.  Since $g,\curl{u},u\cdot \grad\zeta \in L^2\p{\R^+;H^s\p{\T^3;\C^3}}$ and $\omega_0\in H^{1+s}\p{\T^3;\R^3}$, we are in a position to use Theorem \ref{Type 2} to deduce that $\omega\in L^2\p{\R^+;H^{2+s}\p{\T^3;\R^3}}\cap H^1\p{\R^+;H^s\p{\T^3;\R^3}}$ with the estimate, for some universal $K_3\in\R^+$,
\begin{equation} \label{promotion 2 2}
\norm{\omega}_{L^2H^{2+s}}+\norm{\omega}_{H^1H^s}\le K_3\p{\norm{\omega_0}_{H^{1+s}}+\norm{g}_{L^2H^s}+\norm{\curl u}_{L^2\z{H}^s_{\perp}} + \norm{u\cdot \grad\zeta }_{L^2H^s}}.
\end{equation}

The proof is complete upon summing \eqref{promotion 2 1} and \eqref{promotion 2 1} and using Theorem \ref{promotion 1} to bound the right-hand-side in terms of a universal constant times the appropriate norms of the data, forcing, and microtorquing.
\end{proof}

\subsection{Existence of locally integrable solutions to \eqref{mp_linearized}}

We now turn our attention to the question of the existence of locally integrable solutions to \eqref{mp_linearized}.  The low level of temporal regularity and integrability required in the definition of locally integrable solutions makes them amenable to a fairly simple fixed point argument.  We begin by introducing a map used in the fixed point argument.  Note, though, that the requirements of $q$ relative to $s$ are slightly stricter in the following than in Definition \ref{locally integrable solution}.

\begin{defn}\label{contraction def}
Let $s \in \R^+$, 
\begin{equation}
 q \in  
\begin{cases}
 \p{5/2,\infty} &\text{if } s\in\sb{0,3/2} \\
 (s,\infty) &\text{if } s\in\p{3/2,\infty},
\end{cases}
\end{equation}
$T_0 \in \R^+\cup\cb{0}$, and $T \in \R^+$.  Let $\zeta_0\in H^q_\|\p{\T^3;\R^3}$ and let $\zeta$ be a potential microflow starting from $\zeta_0$ as in Definition \ref{potential flow}.  We then define the map
\begin{multline} 
\Psi_{T_0,T} :\z{H}^{1+s}_\perp\p{\T^3;\R^3} \times H^{1+s}\p{\T^3;\R^3} \\
\times \p{ L^2\p{\p{T_0,T_0+T};\z{H}^s_\perp\p{\T^3;\R^3}} \times L^2\p{\p{T_0,T_0+T};H^{s}\p{\T^3;\R^3}}}^2 \\
\to  L^2\p{\p{T_0,T_0+T};\z{H}^s_\perp\p{\T^3;\R^3}}\times L^2\p{\p{T_0,T_0+T};H^{s}\p{\T^3;\R^3}}
\end{multline}
via
\begin{multline} 
\Psi_{T_0,T}\p{u_0,\omega_0,f,g,u,\omega} \\
=
\left.
\p{\mathcal{S}\p{u_0}+\mathcal{S}\ast\p{\f{1}{\varrho}f+\f{\kappa}{\varrho}\curl\omega},\mathcal{T}\p{\omega_0}+\mathcal{T}\ast\p{\f{1}{j}g+\f{\kappa}{j}\curl u-u\cdot \grad\zeta}}
\right\vert_{(T_0,T_0+T)},
\end{multline}
where we have abbreviated $\mathcal{S}=\mathcal{S}_{\f{2\ep+\kappa}{2\varrho}}$ and $\mathcal{T}=\mathcal{T}_{\f{\al}{j},\f{\al+3\be}{3j},\f{\gam}{j},\f{2\kappa}{j}}$ (see Definitions \ref{Heat flow 1}, \ref{Convolution}, and \ref{Type 2 flow and conv}).
\end{defn}

Our next result establishes conditions under which this map is contractive.

\begin{prop}\label{fixed point 1}
Let $s$, $q$, $\zeta_0$, and $\zeta$ be as in Definition \ref{contraction def}, and let $\mu\in\p{0,1}$.  Then there exists a $T\in\R^+$ such that for every time $T_0\in\R^+\cup\cb{0}$, initial data $u_0 \in  \z{H}^{1+s}_\perp\p{\T^3;\R^3}$ and  $\omega_0 \in H^{1+s}\p{\T^3;\R^3}$, forcing $f \in L^2\p{\p{T_0,T_0+T};\z{H}^{s}_\perp\p{\T^3;\R^3}}$, and microtorquing $g \in L^2\p{\p{T_0,T_0+T};H^{s}\p{\T^3;\R^3}}$  it holds that
\begin{multline} \label{fixed point 1 0}
\norm{\Psi_{T_0,T}\p{u_0,\omega_0,f,g,u,\omega}-\Psi_{T_0,T}\p{u_0,\omega_0,f,g,\tilde{u},\tilde{\omega}}}_{L^2\p{(T_0,T_0+T);\z{H}^{s}_\perp}\times L^2\p{(T_0,T_0+T);H^s}} \\
\le \mu\norm{\p{u-\tilde{u},\omega-\tilde{\omega}}}_{L^2\p{(T_0,T_0+T);\z{H}^s_\perp}\times L^2\p{(T_0,T_0+T);H^s}}
\end{multline}
for all pairs of pairs $\p{u,\omega},\p{\tilde u,\tilde\omega}\in L^2\p{\p{T_0,T_0+T};\z{H^s_\perp}\p{\T^3;\R^3}}\times L^2\p{\p{T_0,T_0+T};H^s\p{\T^3;\R^3}}$.
\end{prop}

\begin{proof}
We exhibit the argument for $T_0=0$ only, as the argument for general $T_0\in\R^+$ is similar. Fix $\p{u,\omega}$, $\p{\tilde u,\tilde\omega}$, $\p{u_0,\omega_0,f,g}$ as in the hypotheses.   Observe that for any $T\in\R^+$ we have  
\begin{equation} 
\Psi_{0,T}\p{u_0,\omega_0,f,g,u,\omega}-\Psi_{0,T}\p{u_0,\omega_0,f,g,\tilde u,\tilde\omega}=\Psi_{0,T}\p{0,0,0,0,u-\tilde u,\omega-\tilde\omega}.
\end{equation}
For the sake of brevity write $u-\tilde u=v\in L^2\p{\p{0,T};\z{H}^s_\perp\p{\T^3;\R^3}}$, $\omega-\tilde \omega = \chi \in L^2\p{\p{0,T};H^s\p{\T^3;\R^3}}$, $\Psi\p{v,\chi} = \Psi_{0,T}\p{0,0,0,0,v,\chi}$, and $X=L^2\p{\p{0,T};\z{H}^s_\perp\p{\T^3;\R^3} }\times L^2\p{\p{0,T};H^s\p{\T^3;\R^3}}$. 

We claim that there exists a function $\R^+ \ni T \mapsto  C_T\in\R^+$ (depending on $s$, $q$, and $\zeta_0$) such that $\norm{\Psi\p{v,\chi}}_{X}\le C_T\norm{\p{v,\chi}}_X$ and $\lim_{T\to 0}C_T=0$.  Once the claim is established, the estimate \eqref{fixed point 1 0} follows directly from taking $T$ sufficiently small for given $\mu$.

We now turn to the proof of the claim. First note that 
\begin{equation}
 \Psi\p{v,\chi} = \left.
\p{\mathcal{S}\ast\p{ \f{\kappa}{\varrho}\curl \chi},\mathcal{T}\ast\p{\f{\kappa}{j}\curl v-v\cdot \grad\zeta}}
\right\vert_{(0,T)}.
\end{equation}
To handle the first component of $\Psi\p{v,\chi}$, we let $\al_0=\f{2\ep+\kappa}{2\varrho}$ and use the Cauchy-Schwarz inequality to bound
\begin{multline}
\norm{\S\ast\curl \chi}^2_{L^2\p{\p{0,T}; \z{H}^s_\perp} }  \le \int_{\p{0,T}}\sum_{k\in\Z^3}\abs{k}^{2s}\abs{\int_{\p{0,t}}\exp\p{-4\pi^2\al_0\abs{k}^2\p{t-\tau}}2\pi i k \times \hat{\chi}\p{t,k}\;\m{d}\tau}^2\;\m{d}t \\
\le 4\pi^2 T \sum_{k\in\Z^3} \abs{k}^{2s+2} \p{\int_{\R^+} \exp\p{-8\pi^2\al_0\abs{k}^2 t} \;\m{d}t} \p{\int_{\p{0,T}} \abs{\hat{\chi}\p{t,k}}^2 \;\m{d}t }
\\  
= \f{1}{2\al_0}T\sum_{k\in\Z^3\setminus\cb{0}}\abs{k}^{2s}\int_{\p{0,T}}\;\abs{\hat{\chi}\p{t,k}}^2\m{d}t
=\f{1}{2\al_0}T\norm{\chi}_{L^2\p{\p{0,T}; \z{H}^{s}_\perp}}^2.
\end{multline}
In the same manner, we deduce that up to a universal $c\in\R+$ we may bound:
\begin{equation} 
\begin{cases}
     \norm{\mathcal{T}\ast\curl v}^2_{L^2\p{\p{0,T};H^s }}\le  cT\norm{v}^2_{L^2\p{\p{0,T}; H^s}}\\
     \norm{\mathcal{T}\ast\p{v\cdot \grad\zeta}}_{L^2\p{\p{0,T}; H^{s}} }^2\le cT\norm{v\cdot \grad\zeta}^2_{L^2\p{\p{0,T}; H^{s-1}}}.
\end{cases}
\end{equation}
To handle the latter product term we break to cases based on $s$.  If $0 \le s \le 3/2$, then $s \le 3/2 \le q-1$ and so Proposition \ref{products} allows us to estimate for $c'\in\R^+$ depending on $s$ and $q$ 
\begin{equation}
\norm{v\cdot \grad\zeta}^2_{L^2_T H^{s-1}} \le \norm{v\cdot \grad\zeta }^2_{L^2_T H^{s}} \le c' \norm{v}^2_{L^2\p{\p{0,T}; H^s}} \norm{\grad\zeta}^2_{L^\infty\p{\p{0,T}; H^{q-1}}}. 
\end{equation}
On the other hand, if $3/2 < s$, then $s-1 < q-1$, and so we again use Proposition \ref{products} to obtain $c'\in\R^+$ and bound 
\begin{equation}
\norm{v\cdot \grad\zeta }^2_{L^2\p{\p{0,T}; H^{s-1}}}   \le  c' \norm{v}^2_{L^2\p{\p{0,T}; H^s}} \norm{\grad\zeta}^2_{L^\infty\p{\p{0,T}; H^{s-1}}} \le c' \norm{v}^2_{L^2\p{\p{0,T}; H^s}} \norm{\grad\zeta}^2_{L^\infty\p{\p{0,T}; H^{q-1}}} . 
\end{equation}
In either case we may combine the resulting estimate with Proposition \ref{type 2 flow smoothing}  to see that:
\begin{equation}
\norm{v\cdot \grad \zeta}^2_{L^2\p{\p{0,T}; H^{s-1}}} \le 4\pi^2c' \norm{v}^2_{L^2\p{\p{0,T}; H^s}} \norm{\zeta}^2_{L^\infty\p{\p{0,T}; H^{q}}} \le \tilde{c}   \norm{v}^2_{L^2\p{\p{0,T}; H^s}}  \norm{\zeta_0}_{H^q_\|}^2
\end{equation}
holds for a universal constant $\tilde{c}\in\R^+$.
Hence we may take $C_T$ to be scalar, whose size depends on $\tilde{c}$ and $\norm{\zeta_0}_{H^q_\|}$, multiple of $\sqrt{T}$.
\end{proof}

As a corollary, we can now  produce locally integrable solutions to \eqref{mp_linearized} under slightly stronger hypotheses on $q$ relative to $s$ than stated in Definition \ref{locally integrable solution}.

\begin{coro}\label{existence of locally integrable}
Let $s \in \R^+$, 
\begin{equation}\label{q_strong}
 q \in  
\begin{cases}
 \p{\f{5}{2},\infty} &\text{if } s\in\sb{0,\f32} \\
 [s+1,\infty) &\text{if } s\in\p{\f32,\infty}.
\end{cases}
\end{equation}
Let $\zeta_0\in H^q_\|\p{\T^3;\R^3}$ and let $\zeta$ be the corresponding potential microflow as in Definition \ref{potential flow}.  Then for each data quadruple  
\begin{equation} 
\p{u_0,\omega_0,f,g}\in\z{H}^{1+s}_\perp\p{\T^3;\R^3} \times H^{1+s}\p{\T^3;\R^3} \times L^2\p{\R^+;\z{H}^{s}_\perp\p{\T^3;\R^3}} \times L^2\p{\R^+;H^{s}\p{\T^3;\R^3}}
\end{equation}
there exists a unique pair $(u,\omega)$ that is a locally integrable solution to \eqref{mp_linearized} in the sense of Definition \ref{locally integrable solution}.
\end{coro}

\begin{proof}
Fix $\p{u_0,\omega_0,f,g}$ as in the hypotheses. Proposition \ref{fixed point 1} provides us with a $T_\ast \in\R^+$ such that for all $T_0\in\R^+\cup\cb{0}$ the mapping $\Psi_{T_0,T_\ast}\p{u_0,\omega_0,f,g,\cdot,\cdot}$ is a contraction on $L^2\p{(T_0,T_0+T_\ast);\z{H}^s_{\perp}\p{\T^3;\R^3}}\times L^2\p{(T_0,T_0+T_\ast);H^s\p{\T^3;\R^3}}$.  Thus, for each $n \in \N$ we may apply the Banach fixed point theorem to produce a unique pair $\p{u^{(n)},\omega^{(n)}}$, defined on the temporal interval $(nT_\ast, (n+1)T_\ast )$, such that 
\begin{equation}
\p{u^{(n)},\omega^{(n)}} =  \Psi_{nT_\ast,(n+1)T_\ast}\p{u_0,\omega_0,f,g,u^{(n)},\omega^{(n)}}.
\end{equation}
We then define
\begin{multline}
\p{u,\omega}\in \bigcap_{n\in\N}\p{ L^2\p{\p{nT_\ast,(n+1)T_\ast};\z{H}^s_\perp\p{\T^3;\R^3}} \times L^2\p{\p{nT_\ast,(n+1)T_\ast};H^s\p{\T^3;\R^3}}}\\ 
= \bigcap_{T\in\R^+} \p{ L^2\p{\p{0,T};\z{H}^s_\perp\p{\T^3;\R^3}}\times L^2\p{\p{0,T};H^s\p{\T^3;\R^3}} }
\end{multline}
to be the function equal to $\p{u^{\p{n}},\omega^{\p{n}}}$ on the interval $(nT_\ast, (n+1)T_\ast)$.  By construction, we have that 
\begin{equation}
\p{u,\omega}=\p{\S\p{u_{0}}+\S\ast\p{\f{1}{\varrho}f+\f{\kappa}{\varrho}\curl\omega},\mathcal{T}\p{\omega_{0}}+\mathcal{T}\ast\p{\f{1}{j}g+\f{\kappa}{j}\curl u-u\cdot \grad\zeta}} 
\end{equation}
almost everywhere in $\R^+$. To prove that $(u,\omega)$ is a locally integrable solution, it remains to improve the spatial and local temporal regularity.

Since $\curl\omega\in\bigcap_{T\in\R^+}L^2\p{\p{0,T};\z{H}^{s-1}_\perp\p{\T^3;\R^3} }$, $f\in L^2\p{\R^+;\z{H}^s_\perp\p{\T^3;\R^3}}$, and $u_0\in\z{H}^{1+s}_\perp \p{\T^3;\R^3}$, we are in a position to promote the spatial regularity of $u$ using Propositions \ref{Convolution well defined} and \ref{heat flow 1}.  Applying these  shows that 
\begin{equation}
u\in\bigcap_{T\in\R^+}\p{L^2\p{\p{0,T};\z{H}^{1+s}_\perp\p{\T^3;\R^3}}\cap H^1\p{\p{0,T};\z{H}^{s-1}_\perp\p{\T^3;\R^3}}}. 
\end{equation}
In particular, we learn from this that $\curl u\in\bigcap_{T\in\R^+}L^2\p{\p{0,T};H^s\p{\T^3;\R^3}}$. By the product estimates from Proposition \ref{products} and the constraints on $q$, it is also the case that $u\cdot \grad\zeta$ belongs to this same space;  indeed, for any $T\in\R^+$ we may bound for some $c\in\R^+$ depending on $s$ and $q$
\begin{equation}
\norm{u\cdot \grad\zeta}_{L^2\p{\p{0,T}; H^s}} \le c  \norm{u}_{L^2\p{\p{0,T}; \z{H}^s_\perp}} \norm{\grad \zeta }_{L^\infty\p{\p{0,T}; H^{q-1}}} \le 4\pi^2 c \norm{u}_{L^2\p{\p{0,T}; \z{H}^s_\perp}} \norm{ \zeta }_{L^\infty\p{\p{0,T}; H^{q}}},
\end{equation}
and the latter is finite by the construction of $u$ and Proposition \ref{type 2 flow smoothing}. 

Now, we also have the inclusions $g\in L^2\p{\R^+;H^s\p{\T^3;\R^3}}$ and $\omega_0\in H^{1+s}\p{\T^3;\R^3}$, so  Theorem \ref{Type 2} yields the inclusion 
\begin{equation}
\omega\in\bigcap_{T\in\R^+}L^2\p{\p{0,T};H^{2+s}\p{\T^3;\R^3}}\cap H^1\p{\p{0,T};H^s\p{\T^3;\R^3}}. 
\end{equation}
Finally, we can promote the spatial regularity of $u$ once more to 
\begin{equation}
 u \in \bigcap_{T\in\R^+}L^2\p{\p{0,T};\z{H}^{2+s}_\perp \p{\T^3;\R^3}}\cap H^1\p{\p{0,T};\z{H}^s_\perp\p{\T^3;\R^3}}
\end{equation}
using this new information on the spatial regularity of $\omega$.  Differentiation of the fixed point identity confirms that $\p{u,\omega}$ is a strong solution to \eqref{mp_linearized}.  Hence the pair $\p{u,\omega}$ is indeed a locally integrable solution as in Definition \ref{locally integrable solution}. The uniqueness assertion is now a consequence of Theorem \ref{promotion 1}.

\end{proof}

\subsection{Isomorphism associated to \eqref{mp_linearized}}

We now have all of the tools needed to construct an isomorphism corresponding to the global-in-time well-posedness of \eqref{mp_linearized}.

\begin{thm}\label{linear iso} 
Let $s\in\R^+\cup\cb{0}$ and $q \in \R^+$ satisfy \eqref{q_strong}.  Let $\zeta_0\in H^q_\|\p{\T^3;\R^3}$ and let $\zeta$ be the corresponding potential microflow as in Definition \ref{potential flow}.  Define the map
\begin{multline}
\mathcal{M}:  \p{L^2\p{\R^+;\z{H}^{2+s}_\perp\p{\T^3;\R^3}} \cap  H^1\p{\R^+;\z{H}^s_\perp\p{\T^3;\R^3}}} \\
\times \p{L^2\p{\R^+;H^{2+s}\p{\T^3;\R^3} } \cap H^1\p{\R^+;H^s\p{\T^3;\R^3}}} \\
\to \z{H}^{1+s}_\perp\p{\T^3;\R^3} \times H^{1+s}\p{\T^3;\R^3} \times L^2\p{\R^+;\z{H}^s_\perp \p{\T^3;\R^3}}\times L^2\p{\R^+;H^s\p{\T^3;\R^3}}
\end{multline}
via 
\begin{equation}
 \mathcal{M}(u,\omega) = 
\begin{pmatrix}
      u\p{0}\\
    \omega \p{0}\\
    \varrho \pd_t u- \p{\ep + \frac{\kappa}{2}} \Delta u- \kappa \curl\omega\\
    j( \pd_t\omega +u\cdot \grad\zeta ) -\p{\al+\gam}\Delta\omega - \p{\frac{\al}{3} + \be - \gam}
    \grad\Div\omega + 2\kappa \omega - \kappa \curl u 
\end{pmatrix},
\end{equation}
where $u\p{0}=\lim_{t\to0^+}u\p{t}$ and $\omega\p{0}=\lim_{t\to0^+}\omega\p{t}$ are understood in the $H^{1+s}\p{\T^3;\R^3}$ topology via Proposition \ref{interpolation}.  Then $\mathcal{M}$ is well-defined and yields a bounded linear isomorphism.
\end{thm}
\begin{proof}
We begin by showing that $\mathcal{M}$ is well-defined.  The only term we need to examine is the product term $u\cdot \nabla \zeta$.  The inclusion $u\cdot \nabla \zeta \in  L^2\p{\R^+;H^s\p{\T^3;\R^3}}$ follows directly from the assumptions on $q$ and $s$ and Propositions \ref{products} and \ref{type 2 flow smoothing}:
\begin{equation}
 \norm{u \cdot \nabla \zeta}_{L^2\p{\R^+; H^s}} \le c \norm{u}_{L^2\p{\R^+; H^{s+2}}} \norm{\nabla \zeta}_{L^\infty\p{\R^+; H^{q-1}}} \le c' \norm{u}_{L^2\p{\R^+; H^{s+2}}} \norm{\zeta_0}_{ H^{q}}.
\end{equation}
Where $c,c'\in  \R^+$ are constants depending only on $s$, $q$, and various coefficients of viscosity. Thus $\mathcal{M}$ is well-defined.

Note that $\mathcal{M}(u,\omega) = (u_0,\omega_0,f,g)$ implies that $(u,\omega)$ is a locally integrable solution to \eqref{mp_linearized} with initial data $(u_0,\omega_0)$ and forcing/microtorquing $(f,g)$.  Consequently, the injectivity of $\mathcal{M}$ follows from the estimate in Theorem \ref{promotion 1} and the surjectivity of $\mathcal{M}$ follows from Corollary \ref{existence of locally integrable} and Theorems \ref{promotion 1} and \ref{promotion 2}.  Thus $\mathcal{M}$ is an isomorphism, and the claim is proved.

\end{proof}

\section{Nonlinear analysis}\label{sec_nonlinear}

In this section we first identify the natural nonlinear mapping associated to problem \eqref{micro_polar} and prove that it is globally smooth with respect to our choices of domain and codomain. Then, we show that the linear well-posedness achieved in the previous section is sufficiently strong to prove, with the aid of the inverse function theorem, that around sufficiently regular potential microflows this mapping is in fact a smooth diffeomorphism. Finally, we answer some questions related to the nonlinear stability and attractiveness of the potential microflows.
 
\subsection{Construction of solutions to \eqref{micro_polar} near potential microflows }

We now turn our attention to solving \eqref{micro_polar}.  

\begin{defn}\label{differential operator}
Let $s\in\R^+\cup\cb{0}$.  We define the nonlinear map associated to the problem \eqref{micro_polar} to be
\begin{multline}
\mathcal{Q}:   \p{  L^2\p{\R^+;\z{H}^{2+s}_\perp\p{\T^3;\R^3} }\cap  H^1\p{\R^+;\z{H}^s_\perp\p{\T^3;\R^3}}} \\
\times \p{L^2\p{\R^+;H^{2+s}\p{\T^3;\R^3}}\cap H^1\p{\R^+;H^s\p{\T^3;\R^3}}}\times L^2\p{\R^+;\z{H}^{1+s}\p{\T^3;\R}} \\
\to \z{H}^{1+s}_\perp\p{\T^3;\R^3} \times H^{1+s}\p{\T^3;\R^3} \times L^2\p{\R^+;\z{H}^{s}\p{\T^3;\R^3}} \times L^2\p{\R^+;H^s\p{\T^3;\R^3}}
\end{multline}
given by 
\begin{equation} 
\mathcal{Q}\p{u,\omega,p}=
\begin{pmatrix}
  u\p{0}\\
\omega\p{0}\\
\varrho\p{ \pd_t  + u \cdot \grad }u -\p{\ep + \frac{\kappa}{2}} \Delta u- \kappa \curl\omega  +\grad p\\
j(\pd_t\omega + u \cdot \nabla \omega)   -\p{\al+\gam}\Delta \omega- \p{\frac{\alpha}{3} + \beta - \gamma }   \grad\Div\omega+ 2\kappa \omega - \kappa \curl u
\end{pmatrix}
,
\end{equation}
where $u\p{0}=\lim_{t\to0^+}u\p{t}$ and $\omega\p{0}=\lim_{t\to0^+}\omega\p{t}$ are understood in the $H^{1+s}\p{\T^3;\R^3}$ topology as in Proposition \ref{interpolation}. 
\end{defn}

The following result shows, among other things, that $\mathcal{Q}$ is well-defined.

\begin{prop}\label{diff op well defined}
Let $s \in \R^+ \cup \{0\}$.  Then the map $\mathcal{Q}$ from Definition \ref{differential operator} is well-defined, smooth, and satisfies
\begin{equation} \label{diff op well defined 0}
D\mathcal{Q}\p{v,\chi,q}\p{u,\omega,p}=
\begin{pmatrix}
  u\p{0}\\
\omega\p{0}\\
\varrho\p{ \pd_t u  + v \cdot \grad u + u \cdot \grad v}   -\p{\ep + \frac{\kappa}{2}} \Delta u- \kappa \curl\omega +\grad p\\
j(\pd_t\omega + v \cdot \nabla \omega + u \cdot \nabla \chi)   -\p{\al+\gam}\Delta \omega- \p{\frac{\alpha}{3} + \beta - \gamma }   \grad\Div\omega+ 2\kappa \omega - \kappa \curl u
\end{pmatrix}
\end{equation}
for every $\p{v,\chi,q}$ and $\p{u,\omega,p}$ in the domain of $\mathcal{Q}$.
\end{prop}
\begin{proof}
The well-definedness of the first two components of $\mathcal{Q}$ follows from Proposition \ref{interpolation}.  To complete the proof that $\mathcal{Q}$ is well-defined, we only have to worry about the product terms, as the rest all clearly belong to the codomain.  More precisely, given $\p{u,\omega}$ belonging to the first two components of the domain of $\mathcal{Q}$, we need to prove  that $u\cdot \grad u\in L^2\p{\R^+;\z{H}^s \p{\T^3;\R^3}}$ and $u\cdot \grad\omega \in L^2\p{\R^+;H^s\p{\T^3;\R^3}}$. 

As a first step, we note that $u\cdot \grad u$ has vanishing spatial average as a consequence of $u$ being solenoidal:
\begin{equation} 
\int_{\T^3}u\cdot \grad u =  -\int_{\T^3}u \Div u=0.
\end{equation}
Next we use Propositions \ref{products} and \ref{interpolation} to estimate
\begin{multline} 
\int_{\R^+}\norm{u\p{t} \cdot \grad u\p{t}}^2_{\z{H}^s}\;\m{d}t \le c \int_{\R^+} \norm{u\p{t}}^2_{\z{H}^{2+s}} \norm{u\p{t}}_{\z{H}^{1+s}}^2\;\m{d}t \\
\le c'\norm{u}^2_{L^\infty\p{\R^+;\z{H}^{1+s}}}\norm{u}^2_{L^2\p{\R^+;\z{H}^{2+s}}} \le \tilde{c} \norm{u}_{L^2\p{\R^+;\z{H}^{2+s}}\cap H^1\p{\R^+;\z{H}^s}}^4 < \infty,
\end{multline}
where $c,c',
\tilde{c}\in\R^+$ are constants depending on $s$. A similar  computation shows the inclusion 
\begin{equation}u\cdot \grad\omega \in L^2\p{\R^+;H^s\p{\T^3;\R^3}},.\end{equation} and we conclude that $\mathcal{Q}$ is indeed well-defined.

Next, suppose that $\p{v,\chi,q}$ and $\p{u,\omega,p}$ belong to the domain of $\mathcal{Q}$.  With an abuse of notation, we let $D\mathcal{Q}\p{v,\chi,q}$ be the linear operator defined by \eqref{diff op well defined 0}.  The same arguments used above show that this defines a bounded linear  map.  Then we compute 
\begin{equation} 
\mathcal{Q}\p{v+u,\chi+\omega,q+p}-\mathcal{Q}\p{v,\chi,q}-D\mathcal{Q}\p{v,\chi,q}\p{u,\omega,p}=\p{0,0,u\cdot \grad u,u\cdot \grad\omega},
\end{equation}
and so we may again use the estimates from Propositions  \ref{products} and \ref{interpolation} to deduce that $\mathcal{Q}$ is differentiable with derivative $D\mathcal{Q}$.  The map $\p{v,\chi,q} \mapsto D\mathcal{Q}\p{v,\chi,q}$ is affine and continuous, thus smooth. Hence we conclude that $\mathcal{Q}$ is smooth.

\end{proof}

We now have all of the tools needed to produce solutions to \eqref{micro_polar} with an application of the inverse function theorem.

\begin{thm}\label{inverse function}
Let $s\in\R^+\cup\cb{0}$ and $q \in \R^+$ satisfy \eqref{q_strong}.  Let $\zeta_0\in H^q_\|\p{\T^3;\R^3}$ and let $\zeta$ be the corresponding potential microflow as in Definition \ref{potential flow}.  Then there exist open sets 
\begin{equation}\label{inverse function 00}
\mathcal{W} \subseteq \z{H}^{1+s}_\perp\p{\T^3;\R^3} \times H^{1+s}\p{\T^3;\R^3} \times L^2\p{\R^+;\z{H}^s\p{\T^3;\R^3}}\times L^2\p{\R^+;H^s\p{\T^3;\R^3}} 
\end{equation}
and 
\begin{multline}\label{inverse function 01}
 \mathcal{V} \subseteq    \p{ L^2\p{\R^+;\z{H}^{2+s}_\perp\p{\T^3;\R^3}} \cap  H^1\p{\R^+;\z{H}^s_\perp\p{\T^3;\R^3}}} \\
 \times \p{L^2\p{\R^+;H^{2+s}\p{\T^3;\R^3}} \cap H^1\p{\R^+;H^s\p{\T^3;\R^3}}}\times L^2\p{\R^+;\z{H}^{1+s}\p{\T^3;\R}}
\end{multline}
such that the following hold.  
\begin{enumerate}
 \item $\p{0,\zeta_0,0,0} \in \mathcal{W}$ and $\p{0,\zeta,0}  \in \mathcal{V}$ with $\mathcal{Q}\p{0,\zeta,0}=\p{0,\zeta_0,0,0}$.
 \item $\mathcal{Q}\p{\mathcal{V}} = \mathcal{W}$ and $\mathcal{Q} : \mathcal{V} \to \mathcal{W}$ is a smooth diffeomorphism.
  \item $\mathcal{Q}: \mathcal{V}\to \mathcal{W}$ is bi-Lipschitz, i.e. there exists constants $C_0,C_1\in\R^+$ such that if we abbreviate $X$ and $Y$ for the domain and codomain of $\mathcal{Q}$, respectively, then 
 \begin{equation}\label{stable_equation}
 C_0 \norm{\p{u,\omega,p} - \p{v,\chi,q}}_X \le \norm{\mathcal{Q}\p{u,\omega,p} - \mathcal{Q}\p{v,\chi,q}}_{Y} \le C_1  \norm{\p{u,\omega,p} - \p{v,\chi,q}}_X  
 \end{equation}
 for all $\p{u,\omega,p},\p{v,\chi,q} \in \mathcal{V}\subseteq X$.
  \item For each $(u_0,\omega_0,f,g) \in \mathcal{W}$ the triple $(u,\omega,p) = \mathcal{Q}^{-1}(u_0,\omega_0,f,g) \in \mathcal{V}$ is the unique (in $\mathcal{V}$) strong solution to  \eqref{micro_polar}, achieving the initial data $u_0 = \lim_{t\to 0} u(t)$ and $\omega_0 = \lim_{t\to 0}\omega(t)$ in the $H^{1+s}$ topology as in Proposition \ref{interpolation}.  

\end{enumerate}
\end{thm}
\begin{proof}
Write $X$ and $Y$ for the domain and codomain of $\mathcal{Q}$.  Proposition \ref{diff op well defined} ensures that $\mathcal{Q}: X \to Y$ is smooth and that we have the identity for $\p{u,\omega,p}\in X$ \begin{equation}
D\mathcal{Q}\p{0,\zeta,0}\p{u,\omega,p} = \mathcal{M}\p{u,\omega}+\p{0,0,\grad p,0},
\end{equation}
where $\mathcal{M}$ is the bounded linear isomorphism from Theorem \ref{linear iso}.

To see that this map is injective, suppose that $\p{u,\omega,p}$ belongs to the kernel of $D\mathcal{Q}\p{0,\zeta,0}$. The third component of the equation $\p{0,0,0,0}=D\mathcal{Q}\p{0,\zeta,0}\p{u,\omega,p}$ reads 
\begin{equation}
\varrho\pd_tu+\p{\ep+\f{\kappa}{2}}\Delta u-\kappa\curl\omega=-\grad p.
\end{equation}
At almost every time in $\R^+$ the left hand side belongs to the image of the Leray projector $\P$ over $H^s\p{\T^3;\C^3}$ and the right hand side belongs to the kernel. Since $\P$ induces an orthogonal decomposition $H^s\p{\T^3;\R^3}\cong H^s_\perp\p{\T^3;\R^3}\oplus H^s_\|\p{\T^3;\R^3}$ (see Definition \ref{Leray}), in fact both sides of this equation vanish almost everywhere. Since $\grad: L^2\p{\R^+;\z{H}^{1+s}\p{\T^3;\R}}\to L^2\p{\R^+;H^s_\|\p{\T^3;\R^3}}$ is a bounded linear isomorphism by Proposition \ref{potential_map}, we deduce that $p=0$.  We now learn that $\mathcal{M}\p{u,\omega}=\p{0,0,0,0}$, so  an application of Theorem \ref{linear iso} implies $u=\omega=0$.  Thus, $D\mathcal{Q}\p{0,\zeta,0}$ has trivial kernel and is therefore injective.

We now turn to the proof of surjectivity.  Pick data/forcing/microtorquing $\p{u_0,\omega_0,f,g}\in Y$. Then we have the inclusion
\begin{equation}
    \p{u_0,\omega_0,\P f,g}\in H^{1+s}_\perp\p{\T^3;\R^3}\times H^{1+s}\p{\T^3;\R^3}\times L^2\p{\R^+;\z{H}_\perp^s\p{\T^3;\R^3}}\times L^2\p{\R^+;H^s\p{\T^3;\R^3}}.
\end{equation}
Again using Theorem \ref{linear iso}, we can define $\p{u,\omega}=\mathcal{M}^{-1}\p{u_0,\omega_0,\P f,g}$.  We then use Proposition \ref{potential_map} to define $p=\Pi\p{I-\P}f$.  It follows that $\p{u,\omega,p}\in X$, so
\begin{multline}
    D\mathcal{Q}\p{0,\zeta,0}\p{u,\omega,p}=\mathcal{M}\mathcal{M}^{-1}\p{u_0,\omega_0,\P f,g}+\p{0,0,\grad\Pi\p{I-\P}f,0} \\
    =\p{u_0,\omega_0,\P f,g} + \p{u_0,\omega_0,(I-\P)f,g}
    =\p{u_0,\omega_0,f,g},
\end{multline}
and surjectivity is proved.  

We now know that $D \mathcal{Q}(0,\zeta,0)$ is an isomorphism.  Consequently, the inverse function theorem (see, for instance, Theorem 2.5.2 in~\cite{abraham_marsden_ratiu}) provides for the existence of open sets $\mathcal{W} \subseteq Y$ and $\mathcal{V} \subseteq X$ satisfying the first three stated results.  The fourth item then follows from the second item and the definition of $\mathcal{Q}$.

\end{proof}

\subsection{Stability and asymptotic stability}
 
We now turn our attention to the stability and attractiveness of the potential microflows and nearby solutions.

\begin{thm}\label{s_a}
Let $s,q\in\R^+\cup\cb{0}$ and $\zeta_0,\zeta$ be as in Theorem \ref{inverse function}.  Let $X$ and $Y$ denote the domain and codomain, respectively, of the map $\mathcal{Q}$ from Definition \ref{differential operator}.  Let $\mathcal{W} \subseteq Y$ and $\mathcal{V} \subseteq X$ be the open sets from Theorem \ref{inverse function}.  Fix $\p{u_0,\omega_0,f,g} \in \mathcal{W}$ and let $\p{u,\omega,p} \in \mathcal{V}$ be the associated solution to \eqref{micro_polar} given by Theorem \ref{inverse function}.  Then the following hold.
\begin{enumerate}

\item \textbf{Stability}:  There exists a universal constant $B \in \R^+$ such that if $\p{v_0,\chi_0,\varphi,\psi}\in \mathcal{W}$ and $\p{v,\chi,q} \in \mathcal{V}$ is the associated solutions to \eqref{micro_polar} given by Theorem \ref{inverse function}, then  
\begin{equation}\label{s_a_00}
    \norm{\p{u-v,\omega-\chi,p-q}}_X \le B\norm{\p{u_0-v_0,\omega_0-\chi_0,f-\varphi,g-\psi}}_Y.
\end{equation}
In particular, the solution $(u,\omega,p)$ is stable.

\item \textbf{Attractiveness:} We have that 
\begin{equation}\label{s_a_01}
\lim_{t \to \infty} u\p{t} =0 \text{ in } \z{H}^{1+s}_\perp\p{\T^3;\R^3} \text{ and } 
\lim_{t \to \infty} \p{\omega\p{t}-\zeta\p{t}} =0 \text{ in } H^{1+s}\p{\T^3;\R^3}.
\end{equation}
Moreover, if $s\in\p{1/2,\infty}$, then
\begin{equation}\label{s_a_02}
    \lim_{t\to\infty}p\p{t}=0\text{  in  }\;\z{H}^{1+s}\p{\T^3;\R}\quad\lra\quad\lim_{t\to\infty}\p{I-\P}f\p{t}=0\text{  in  }H^s_\|\p{\T^3;\R^3}.
\end{equation}

\item \textbf{Exponential decay without forcing:}  Suppose that $f=g=0$.  Then there exists a universal constant $\lambda \in \R^+$ and a constant $\mathcal{K} \in \R^+$, depending on $\zeta_0$, $s$, $q$,  and the various physical parameters, such that for all $\vartheta \in [0,1]$ we have the exponential decay estimate 
\begin{equation}\label{s_a_05}
\norm{u\p{t}}^2_{\z{H}^{\p{1-\vartheta}\p{1+s}}_\perp}+\norm{\omega\p{t}-\zeta\p{t}}^2_{H^{\p{1-\vartheta}\p{1+s}}}  \le  
\mathcal{K} \exp\p{-\lambda \vartheta t} \p{ \norm{u_0}_{\z{H}^{1+s}_\perp}^2+\norm{\omega_0-\zeta_0}^2_{H^{1+s}}   }
\end{equation}
for almost every $t \in \R^+$.

\item \textbf{Algebraic decay with forcing:}  Suppose the forcing and microtorquing decay algebraically in the sense that there exist $\mu,C\in\R^+$ such that
\begin{equation}\label{s_a_alg}
 \int_{\T^3}\abs{f\p{t}}^2+\abs{g\p{t}}^2\le\f{C}{\p{1+t}^\mu} \text{ for a.e. }t \in \R^+.
\end{equation}
Then there exists a constant $\mathcal{K} \in \R^+$, depending on $\norm{u_0}_{\z{H}^{1+s}_\perp},$ $\norm{\omega_0-\zeta_0}^2_{H^{1+s}}$,  $\norm{f}_{L^2\z{H}^s}^2,$ $\norm{g}^2_{L^2H^s}$, $\mu,$ $C,$ $s,$ $q,$ and the various physical parameters, such that  for all $\vartheta\in\sb{0,1}$ we have the algebraic decay estimate 
\begin{equation}\label{s_a_03}
\norm{u\p{t}}^2_{\z{H}^{\p{1-\vartheta}\p{1+s}}_\perp}+\norm{\omega\p{t}-\zeta\p{t}}^2_{H^{\p{1-\vartheta}\p{1+s}}}   
\le \frac{\mathcal{K}}{(1+t)^{\mu \vartheta}}  \text{ for almost every } t \in \R^+.
\end{equation}

\item \textbf{Exponential decay with forcing:} Suppose the forcing and microtorquing decay exponentially in the sense that there exist $\mu, C\in\R^+$ such that
\begin{equation}\label{s_a_exp}
 \int_{\T^3}\abs{f\p{t}}^2+\abs{g\p{t}}^2\le C \exp\p{-\mu t} \text{ for a.e. }t \in \R^+.
\end{equation}
Then there exists a constant $\lambda \in \R^+$, depending on $\mu$ and the physical parameters, and a constant $\mathcal{K} \in \R^+$, depending on $\norm{u_0}_{\z{H}^{1+s}_\perp},$ $\norm{\omega_0-\zeta_0}^2_{H^{1+s}}$,  $\norm{f}_{L^2\z{H}^s}^2,$ $\norm{g}^2_{L^2H^s}$, $\mu,$ $C,$ $s,$ $q,$ and the various physical parameters, such that for all $\vartheta\in\sb{0,1}$ we have the exponential decay estimate
\begin{equation}\label{s_a_04}
\norm{u\p{t}}^2_{\z{H}^{\p{1-\vartheta}\p{1+s}}_\perp}+\norm{\omega\p{t}-\zeta\p{t}}^2_{H^{\p{1-\vartheta}\p{1+s}}}
\le \mathcal{K} \exp\p{-\lambda \vartheta t} 
\end{equation}
for almost every $t \in \R^+$.

\end{enumerate}
\end{thm}
\begin{proof}
The first item is a mere rephrasing of \eqref{stable_equation} from Theorem \ref{inverse function}. 

We now prove the second item.  The assertions in  \eqref{s_a_01} follow from the inclusion
\begin{equation} 
\p{u,\omega-\zeta}\in \p{L^2\p{\R^+;\z{H}^{2+s}_\perp}\cap H^1\p{\R^+;\z{H}^s_\perp}}\times\p{L^2\p{\R^+;H^{2+s}}\cap H^1\p{\R^+;H^s}}
\end{equation}
and the second item of  Proposition \ref{interpolation}.   Next we prove \eqref{s_a_02}.  Applying $\p{I-\P}$ to the third component of the identity $\mathcal{Q}\p{u,\omega,p}=\p{u_0,\omega_0,f,g}$ reveals that
\begin{equation}
\grad p=\p{I-\P}\p{f-u\cdot\grad u}. 
\end{equation}
Hence, by Propositions \ref{potential_map} and \ref{products} we find that up to suitable constants $A_0,A_1,A_1\in\R^+$ (independent of time) it holds for almost every $t\in\R^+$:
\begin{multline}
    \norm{p\p{t}}_{\z{H}^{1+s}}=\norm{\Pi\p{I-\P}\p{f\p{t}-u\p{t}\cdot\grad u\p{t}}}_{\z{H}^{1+s}}\\\le A_0  \p{\norm{\p{I-\P}f\p{t}}_{\z{H}^s_\|}+\norm{u\p{t}\cdot\grad u\p{t}}_{\z{H}^s}}
    \le A_1\p{\norm{\p{I-\P}f\p{t}}_{\z{H}^{1+s}_\|}+\norm{u\p{t}}_{\z{H}^{1+s}}^2}
\end{multline}
and
\begin{equation}
    \norm{\p{I-\P}f\p{t}}_{\z{H}^s_\|}=\norm{\grad p\p{t}+\p{I-\P}\p{u\p{t}\cdot\grad u\p{t}}}_{\z{H}^s}\le A_2 \p{\norm{p\p{t}}_{\z{H}^{1+s}}+\norm{u\p{t}}_{\z{H}^{1+s}}^2}.
\end{equation}
In light of \eqref{s_a_01},  these estimates imply the equivalence asserted in \eqref{s_a_02}.  This proves the second item.

We now turn to the proof of the decay estimates in the third, fourth, and fifth items.  We know from Theorem \ref{inverse function} that $(0,\zeta,0) \in \mathcal{V}$ is the solution corresponding to $(0,\zeta_0,0,0) \in \mathcal{W}$.  From this, \eqref{s_a_00}, and Proposition \ref{interpolation} we then find that there is a universal constants $\mathcal{K}_1\in\R^+$ for which we have the bound
\begin{equation}\label{s_a_1}
\norm{u\p{t}}_{\z{H}^{1+s}_\perp}^2+\norm{\omega\p{t}-\zeta\p{t}}^2_{H^{1+s}}
\\
\le 
\mathcal{K}_1\p{\norm{u_0}_{\z{H}^{1+s}_\perp}^2+\norm{\omega_0-\zeta_0}^2_{H^{1+s}}+\norm{f}_{L^2\z{H}^s}^2+\norm{g}^2_{L^2H^s}}
\end{equation}
for $t \in \R^+$.

We next study the time evolution of the $L^2$ norms.  We have $\mathcal{Q}\p{u,\omega-\zeta,p}=\p{ u_0,\omega_0-\zeta_0,f,g-u\cdot\grad\zeta}$.  With the aid of the Leray projector, this tells us that the pair $\p{u,\omega-\zeta}$ is a strong solution to the initial value problem
\begin{equation} 
\begin{cases}
\Div u=0 &\text{in } \R^+\times\T^3 \\
\varrho \pd_tu-\p{\ep + \frac{\kappa}{2}}\Delta u - \kappa u+ \varrho \P\p{u\cdot\grad u}= f &\text{in }\R^+\times\T^3\\
j \pd_t\p{\omega-\zeta}-\p{\al+\gam}\Delta\p{\omega-\zeta}-\p{\frac{\al}{3}+\be-\gam}\grad\Div\p{\omega-\zeta} \\
\quad + 2\kappa \p{\omega-\zeta}-\kappa\curl u + j u\cdot\grad\p{\omega-\zeta}
= g- j u\cdot\grad\zeta &\text{in } \R^+\times\T^3  \\
\p{u\p{0},\omega\p{0}-\zeta\p{0}}=\p{u_0,\omega_0-\zeta_0} &\text{on }T^3.
\end{cases}
\end{equation}
Note that integration by parts shows 
\begin{equation}
\int_{\T^3}\P\p{u\cdot\grad u}\cdot u=
\int_{\T^3}\p{u\cdot\grad u}\cdot u= -\int_{\T^3}\p{u\cdot\grad u}\cdot u 
\end{equation}
and 
\begin{equation}
 \int_{\T^3}\p{u\cdot
\grad\p{\omega-\zeta}}\cdot\p{\omega-\zeta}= -\int_{\T^3}\p{u\cdot\grad\p{\omega-\zeta}}\cdot\p{\omega-\zeta},
\end{equation}
and hence all of these quantities vanish.  We may then argue as in Lemmas \ref{ED} and \ref{ED estimates}  to arrive at the following differential inequality for almost every $t\in\R^+$:
\begin{multline}
\p{\int_{\T^3}\f{\varrho}{2}\abs{u\p{t}}^2+\f{j}{2} \abs{\omega\p{t}-\zeta\p{t}}^2}' + \mathcal{C}_0\int_{\T^3}\f{\varrho}{2}\abs{u\p{t}}^2+\f{j}{2} \abs{\omega\p{t}-\zeta\p{t}}^2 + \mathcal{C}_1\norm{\p{u\p{t},\omega\p{t}}}^2_{\z{H}^{1}_\perp\times H^1} \\
\le\int_{\T^3}-\p{u\p{t}\cdot\grad\zeta\p{t}}\cdot\p{\omega\p{t}-\zeta\p{t}} +  u\p{t}\cdot f\p{t}+\p{\omega\p{t}-\zeta\p{t}}\cdot g\p{t}. 
\end{multline}
Let 
\begin{equation}
 \mathcal{C}_0=\min\cb{\f{\pi^2\ep}{\varrho},\f{\ep\kappa}{2j\p{\ep+\kappa}}}, \quad  \mathcal{C}_1=\min\cb{\f{\pi^2\ep}{2},\f{\ep\kappa}{4\p{\ep+\kappa}},2\al,3\be,2\gam},
\end{equation}
and $\mathcal{C}_2$ be an embedding constant from $H^{q-1}\p{\T^3;\R^3}\emb L^\infty\p{\T^3;\R^3}$ (note $q-1>\f32$).  Using these, Cauchy's inequality, and the properties of $\zeta$, we may then estimate the forcing term:
\begin{multline}
\int_{\T^3}-\p{u\p{t}\cdot\grad\zeta\p{t}}\cdot\p{\omega\p{t}-\zeta\p{t}} +  u\p{t}\cdot f\p{t}+\p{\omega\p{t}-\zeta\p{t}}\cdot g\p{t} \\ \le \f{\mathcal{C}_0}{4}\int_{\T^3}\f{j}{2}\abs{\omega\p{t}-\zeta\p{t}}^2+\f{2}{j\mathcal{C}_0} \norm{\grad\zeta\p{t}}^2_{L^\infty \p{\T^3;\R^3}}\int_{\T^3}\abs{u\p{t}}^2 \\  
+\f{\mathcal{C}_0}{4}\int_{\T^3}\f{\varrho}{2}\abs{u\p{t}}^2+\f{j}{2}\abs{\omega\p{t}}^2+\f{2}{\mathcal{C}_0}\max\cb{\f{1}{\varrho},\f{1}{j}}\int_{\T^3}\abs{f\p{t}}^2 + \abs{g\p{t}}^2 \\
\le \f{\mathcal{C}_0}{2}\p{\int_{\T^3}\f{\varrho}{2}\abs{u\p{t}}^2 + \f{j}{2} \abs{\omega\p{t}-\zeta\p{t}}^2} \\ 
+\f{4\pi^2\mathcal{C}_2}{j\mathcal{C}_0} \norm{\zeta\p{t}}^2_{H^q\p{\T^3;\R^3}} \int_{\T^3}\abs{u\p{t}}^2+\f{2}{\mathcal{C}_0}\max\cb{\f{1}{\varrho},\f{1}{j}}\int_{\T^3}\abs{f\p{t}}^2+\abs{g\p{t}
}^2  \\
\le\f{\mathcal{C}_0}{2}\p{\int_{\T^3}\f{\varrho}{2}\abs{u\p{t}}^2+\f{j}{2}\abs{\omega\p{t}-\zeta\p{t}}^2} \\+ 
\f{4\pi^2\mathcal{C}_2}{j\mathcal{C}_0} \exp\p{-\f{4\kappa}{j}t}\norm{\zeta_0}^2_{H^q}\int_{\T^3}\abs{u\p{t}}^2+\f{2}{\mathcal{C}_0}\max\cb{\f{1}{\varrho},\f{1}{j}}\int_{\T^3}\abs{f\p{t}}^2+\abs{g\p{t}
}^2.
\end{multline}

Now define $\mathcal{E}, \xi,\mathcal{F} : \R^+ \to \R$ via 
\begin{equation}
\mathcal{E}\p{t}=\int_{\T^3}\f{\varrho}{2}\abs{u\p{t}}^2+\f{j}{2}\abs{\omega\p{t}-\zeta\p{t}}^2, \quad  \xi\p{t}=\f{\mathcal{C}_0}{2}-\f{4\pi^2\mathcal{C}_2}{\varrho j\mathcal{C}_0}\exp\p{-\f{4\kappa}{j}t}\norm{\zeta_0}^2_{H^q}
\end{equation}
and 
\begin{equation}
\mathcal{F}\p{t}=\f{2}{\mathcal{C}_0}\max\cb{\f{1}{\varrho},\f{1}{j}}\int_{\T^3}\abs{f\p{t}}^2+\abs{g\p{t}}^2. 
\end{equation}
The previous two estimate then imply that 
\begin{equation}\label{s_a_2}
 \mathcal{E}'+\xi\mathcal{E}\le\mathcal{F} \text{ in } \R^+.
\end{equation}
Applying Gronwall's lemma to \eqref{s_a_2}, we find that for almost every $t\in\R^+$ we have the bound
\begin{multline}\label{s_a_3}
\int_{\T^3}\abs{u\p{t}}^2+\abs{\omega\p{t}-\zeta\p{t}}^2  
\le \exp\p{\f{\mathcal{C}_2\pi^2}{4\mathcal{C}_0\kappa\varrho}\norm{\zeta_0}^2_{H^q}} \\
\times \p{\f{\max\cb{\varrho,j}}{\min\cb{\varrho,j}}\exp\p{-\f{\mathcal{C}_0}{2}t} \p{\int_{\T^3}\abs{u_0}^2+\abs{\omega_0-\zeta_0}^2} 
+  \int_{\p{0,t}}\exp\p{-\f{\mathcal{C}_0}{2}\p{t-s}}\mathcal{F}\p{s}\;\m{d}s}.
\end{multline}

We now break to cases based on the decay assumptions of the force and microtorque.  If $f=g=0$, then \eqref{s_a_05} follows by combining \eqref{s_a_1}, \eqref{s_a_3}, and the Sobolev interpolation bound for $h\in H^{1+s}\p{\T^3;\R^3}$
\begin{equation}\label{s_a_interp}
 \norm{h}_{H^{(1-\vartheta)(1+s)}} \le  \norm{h}_{H^{1+s}}^{1-\vartheta}  \norm{h}_{L^{2}}^{\vartheta}.
\end{equation}

Now suppose that $f,g$ decay algebraically, i.e. \eqref{s_a_alg} holds.  Then we can choose $D \in \R^+$ such that $\mathcal{F}\p{t}\le D\p{1+t}^{-\mu}$ for almost every $t\in\R^+$.  Using a change of coordinates, we can then estimate the convolution-like term appearing above as follows:
\begin{equation}
    \int_{\p{0,t}}\exp\p{-\f{\mathcal{C}_0}{2}\p{t-s}}\mathcal{F}\p{s}\;\m{d}s\le D\exp\p{-\f{\mathcal{C}_0}{2}t}\int_{\p{1,\exp\p{\f{\mathcal{C}_0}{2
    }t}}}\f{1}{\p{1+\f{2}{\mathcal{C}_0}\log\p{r}}^\mu}\;\m{d}r.
\end{equation}
The right-most integrand is convex over the interval of integration, so we can bound it from above with a simple two-trapezoid estimate:
\begin{multline}
    \int_{\p{1,\exp\p{\f{\mathcal{C}_0}{2}t}}}\f{1}{\p{1+\f{2}{\mathcal{C}_0}\log\p{r}}^\mu}\;\m{d}r 
    \le \f12\p{\exp\p{\f{\mathcal{C}_0}{4}t}-1}\p{\f{1}{\p{1+\f{t}{2}}^\mu}+1} \\
    +\f12\p{\exp\p{\f{\mathcal{C}_0}{2}t}-\exp\p{\f{\mathcal{C}_0}{4}t}}\p{\f{1}{\p{1+t}^\mu}+\f{1}{\p{1+\f{t}{2}}^\mu}}\\
    \le\exp\p{\f{\mathcal{C}_0}{4}t}+\f{2^\mu+1}{2}\exp\p{\f{\mathcal{C}_0}{2}t}\f{1}{\p{1+t}^\mu}\le\exp\p{\f{\mathcal{C}_0}{2}t}\p{\mathcal{C}_3+\f{2^\mu+1}{2}}\f{1}{\p{1+t}^\mu},
\end{multline}
where 
\begin{equation}
 \mathcal{C}_3=\sup_{t\in\R^+}\p{1+t}^\mu\exp\p{-\f{\mathcal{C}_0}{4}t}\in\R^+.
\end{equation}
Hence, if we let $\mathcal{C}_4=D\p{\mathcal{C}_3+\f{2^\mu+1}{2}}$, then we can combine these bounds with \eqref{s_a_3} to see that for $t\in\R^+$,
\begin{multline}
    \int_{\T^3}\abs{u\p{t}}^2+\abs{\omega\p{t}-\zeta\p{t}}^2 \\
    \le\exp\p{\f{\mathcal{C}_2\pi^2}{4\mathcal{C}_0\kappa\varrho}\norm{\zeta_0}^2_{H^q}}\p{\f{\max\cb{\varrho,j}}{\min\cb{\varrho,j}}\exp\p{-\f{\mathcal{C}_0}{2}t}\int_{\T^3}\abs{u_0}^2+\abs{\omega_0-\zeta_0}^2+\f{\mathcal{C}_4}{\p{1+t}^\mu}}.
\end{multline}
Then \eqref{s_a_03}  follows by combining this with \eqref{s_a_1} and \eqref{s_a_interp}.

Finally, we handle the exponentially decaying case, in which case \eqref{s_a_exp} holds.  Then we can pick $D \in \R^+$ such that $\mathcal{F}\p{t}\le D\exp\p{-\mu t}$ for $t\in\R^+$. Now we can explicitly compute 
\begin{multline}
    \int_{\p{0,t}}\exp\p{-\f{\mathcal{C}_0}{2}\p{t-s}}\mathcal{F}\p{s}\;\m{d}s
    \le D\exp\p{-\f{\mathcal{C}_0}{2}t}\int_{\p{0,t}}\exp\p{\p{\f{\mathcal{C}_0}{2}-\mu}s}\;\m{d}s \\
    =
\begin{cases}
     Dt\exp\p{-\f{\mathcal{C}_0}{2}t}& \text{if }\f{\mathcal{C}_0}{2}=\mu \\
     D\f{\exp\p{-\f{\mathcal{C}_0}{2}t}-\exp\p{-\mu t}}{\mu-\f{\mathcal{C}_0}{2}}&\text{if } \f{\mathcal{C}_0}{2}\neq \mu.
\end{cases}
\end{multline}
Using this and the mean-value theorem, we may then argue as above to deduce that \eqref{s_a_04} holds.

\end{proof}

\appendix

\section{Reduction to velocity fields with vanishing average}\label{app_reduction}
The natural setting for the initial data, force, and microtorque in \eqref{micro_polar} is the space 
\begin{equation}
    H^{1+s}_\perp\p{\T^3;\R^3}\times H^{1+s}\p{\T^3;\R^3}\times L^2\p{\R^+;\z{H}^s\p{\T^3;\R^3}}\times L^2\p{\R^+;H^s\p{\T^3;\R^3}} 
    \text{ for } s\in\R^+\cup\cb{0}.
\end{equation}
However, as we will see below, the corresponding solutions are globally integrable in time if and only if $u_0$ has vanishing average.  It is thus convenient to introduce a change of unknowns that allows us to reduce to studying this case.  This is possible due to the invariance of the micropolar equations \eqref{micro_polar} under Galilean transformations.

\begin{lem}\label{WLOG average zero}
Let $s\in\R^+\cup\cb{0}$. Suppose that $u_0,\omega_0\in H^{1+s}\p{\T^3;\R^3}$ are initial data with $\Div u_0=0$, $f,g\in L^2\p{\R^+;H^s\p{\T^3;\R^3}}$ are forcing and microtorquing with $\int_{\T^3}f\p{t,x}\;\m{d}x=0$ for almost every $t\in\R^+$. Suppose that $\p{u,\omega,p}$ are a corresponding strong solution triple to system \eqref{micro_polar}. Then the following hold.
\begin{enumerate}
\item For almost every $t\in\R^+$ it holds that $\int_{\T^3}u\p{t,x}\;\m{d}x=\int_{\T^3}u_0\p{x}\;\m{d}x$.
\item Given any $b\in\R^3$, the triple $\p{v,\chi,q}: \R^+ \times \T^3 \to \R^3 \times \R^3 \times \R$ defined by  
\begin{equation}
    \begin{cases}
    v\p{t,x}=-b+u\p{t,x+tb}\\\chi\p{t,x}=\omega\p{t,x+tb}\\q\p{t,x}=p\p{t,x+tb}\\\varphi\p{t,x}=f\p{t,x+tb}\\\psi\p{t,x}=g\p{t,x+tb}
    \end{cases} 
\end{equation}
is a strong solution to system \eqref{micro_polar} with initial data $\p{u_0-b,\omega_0}$ and forcing/microtorquing $\p{\varphi,\psi}$.  Moreover, if we posit the space-time regularity of $\p{u,\omega,p}$ to be encoded with the inclusion:
\begin{multline}
    \p{u-\int_{\T^3}u_0,\omega,p}\in\p{L^2\p{\R^+;\z{H}^{2+s}_\perp\p{\T^3;\R^3}}\cap H^1\p{\R^+;\z{H}^s_\perp\p{\T^3;\R^3}}}\\\times\p{L^2\p{\R^+;H^{2+s}\p{\T^3;\R^3}}\cap H^1\p{\R^+;H^s\p{\T^3;\R^3}}}\times\p{L^2\p{\R^+;\z{H}^{1+s}\p{\T^3;\R^3}}},
\end{multline} 
then the same inclusion is true for $\p{v-\int_{\T^3}\p{u_0-b},\chi,q}$.  Also, $\p{\varphi,\psi}$ belong to the same space-time regularity class as $\p{f,g}$.
\end{enumerate}
As a consequence, to understand the solvability of \eqref{micro_polar} with respect to data and forcing / microtorquing quadruples belonging to the space 
\begin{equation}
H^{1+s}_\perp\p{\T^3;\R^3}\times H^{1+s}\p{\T^3;\R^3}\times L^2\p{\R^+;\z{H}^s\p{\T^3;\R^3}}\times L^2\p{\R^+;H^s\p{\T^3;\R^3}}, 
\end{equation}
it is sufficient to understand the system's solvability for data/forcing/microtorquing belonging to the smaller space 
\begin{equation}
\z{H}^{1+s}_\perp\p{\T^3;\R^3}\times H^{1+s}\p{\T^3;\R^3}\times L^2\p{\R^+;\z{H}^s\p{\T^3;\R^3}}\times L^2\p{\R^+;H^s\p{\T^3;\R^3}}. 
\end{equation}
\end{lem}
\begin{proof}
Suppose that we are given a strong solution triple $\p{u,\omega,p}$ corresponding to the data / forcing / microtorquing quadruple $\p{u_0,\omega_0,f,g}$ as in the hypotheses. Averaging the second equation in \eqref{micro_polar}, and then integration by parts yields the identity for almost everywhere on $\R^+$:
\begin{equation}
    0=\int_{\T^3}\varrho\p{\pd_tu+u\cdot\grad u}-\p{\ep+\f{\kappa}{2}}\Delta u-\kappa\curl\omega+\grad p=\varrho\p{\int_{\T^3}u}'+\int_{\T^3}u\Div u=\varrho\p{\int_{\T^3}u}'.
\end{equation}
Thus, the first item now follows from the fundamental theorem of calculus.

Now, if $\mu\in\N^2$ is a multi-index with $\abs{\mu}\le 2$, then for almost every $\p{t,x}\in\R^+\times\T^3$ we have that  $\pd^\mu v\p{t,x}=\pd^\mu u\p{t,x+tb}$ and $\pd^\mu \chi\p{t,x}=\pd^\mu\omega\p{t,x+tb}$, and $\grad q\p{t,x}=\grad p\p{t,x+tb}$.  We next compute the discrepancy in the time derivatives: $\pd_tv\p{t,x}=\pd_tu\p{t,x+tb}+b\cdot\grad u\p{t,x+tb}$, and  $\pd_t\chi\p{t,x}=\pd_t\omega\p{t,x}+b\cdot\grad\omega\p{t,x+tb}$. Finally, using the previous remarks we verify the following equality  between the two material derivative terms:
\begin{equation}
    \begin{cases}
     \p{\pd_t+v\p{t,x}\cdot\grad}v\p{t,x}=\pd_tu\p{t,x+tb}+u\p{t,x+tb}\cdot\grad u\p{t,x+tb}\\
          \p{\pd_t+v\p{t,x}\cdot\grad}v\p{t,x}=\pd_t\omega\p{t,x+tb}+u\p{t,x+tb}\cdot\grad \omega\p{t,x+tb}.
    \end{cases}
\end{equation}
Thus, evaluation of the system \eqref{micro_polar} at points $\p{t,x+tb}$ for $\p{t,x}\in\R^+\times\T^3$ yields that the triple $\p{v,\chi,q}$ is a strong solution to \eqref{micro_polar} with data $\p{u_0-b,\omega_0}$ and forcing/microtorquing $\p{\varphi,\psi}$. The space-time regularity assertions in the latter part of the second item are now a trivial verification given that the first item shows that the velocity's average is constant in time. The consequence follows by: changing coordinates and unknowns, taking $b$ equal to the spatial average of $u_0$; solving the new system using the hypothesized solvability of \eqref{micro_polar} in the average-zero spaces; and finally transforming back with $b$ equal to the negative of the average of $u_0$ and citing the conclusion of the second item.
\end{proof}

\section{Tools from analysis}\label{app_tools}

In this section we record a number of analysis results used throughout the paper.

\subsection{Real-valued distributions and Fourier coefficients}

The following lemma characterizes when a distribution $T \in \D^\ast\p{\T^d;\C^\ell}$ is actually $\R^\ell-$valued.  To state the result we recall a few definitions.  First, for a sequence $z: \Z^d \to \C^\ell$ we define its reflection $Rz : \Z^d \to \C^\ell$ via $Rz(k) = z(-k)$ for $k\in \Z^d$. Second, for a distribution $T \in \D^\ast\p{\T^d;\C^\ell}$ we define its complex conjugate as the distribution $\bar{T}\in \D^\ast\p{\T^d;\C^\ell}$ given by  $\br{\bar{T}, \psi} = \Bar{\br{T,\bar{\psi}}}$ for each $\psi \in\mathcal{D}(\T^d;\C)$.

\begin{lem}\label{real_lemma}
Let $\ell,d \in \N^+$
The following hold.
\begin{enumerate}
 \item If $f \in L^2(\T^d;\C^\ell)$, then $f$ is $\R^\ell$-valued, i.e. $f = \bar{f}$, if and only if $\hat{f} \in \ell^2(\Z^d;\C^\ell)$ satisfies $\Bar{\hat{f}} = R \hat{f}$.

 \item   If $T \in \D^\ast\p{\T^d;\C^\ell}$, then $T$ is $\R^\ell$-valued, i.e. $T = \bar{T}$, if and only if $\overline{\hat{T}} = R \hat{T}$ holds on the lattice $\Z^d$.

\end{enumerate}
\end{lem}
\begin{proof}
If $f = \bar{f}$, then we have that
\begin{equation}
 \overline{\hat{f}(k)} = \int_{\T^d} \overline{f(x)} 
 \bf{e}_k\p{x}\;\m{d}x = \int_{\T^d} f(x) \bf{e}_k\p{x}\m{d}x = \hat{f}(-k),
\end{equation}
and if $\bar{\hat{f}} = R \hat{f}$, then for almost every $x\in\T^d$
\begin{equation}
    \overline{f(x)} = \sum_{k \in \Z^d} \overline{\hat{f}(k)}\bf{e}_{-k}\p{x}  = \sum_{k \in \Z^d} \hat{f}(-k)\bf{e}_{-k}\p{x} = \sum_{k \in \Z^d} \hat{f}(k)\bf{e}_{k} = f(x).
\end{equation}
This proves the first item.

We now turn to the proof of the second item.  If $\psi\in\mathcal{D}\p{\T^d;\C}$, then the series $\sum_{k\in\Z^d}\hat{\psi}\p{k}\bf{e}_k$ converges absolutely in the topology of $H^s\p{\T^3;\C^\ell}$ for all $s\in\R$. In particular, for each $m\in\N$ we have the convergence:
\begin{equation}
\lim_{K\to\infty}\sb{\psi-\sum_{\substack{k\in\Z^d\\\abs{k}\le K}}\hat{\psi}\p{k}\bf{e}_k}_m=0,
\end{equation}
where $\sb{\cdot}_m$ is the seminorm from Definition \ref{space of test functions}.  Thus, given $T\in\mathcal{D}^\ast\p{\T^d;\C^\ell}$ such that $\Bar{\hat{T}}=R\hat{T}$, we can compute its action on $\psi$ via the series:
\begin{multline}
    \br{T,\psi}=\\\lim_{K\to\infty}\sum_{\substack{k\in\Z^d\\\abs{k}\le K}}\hat{\psi}\p{k}\br{T,\bf{e}_k}=\lim_{K\to\infty}\sum_{\substack{k\in\Z^d\\\abs{k}\le K}}\hat{\psi}\p{k}\hat{T}\p{-k}=\lim_{K\to\infty}\sum_{\substack{k\in\Z^d\\\abs{k}\le K}}\hat{\psi}\p{k}
R\hat{T}\p{k}=\lim_{K\to\infty}\sum_{\substack{k\in\Z^d\\\abs{k}\le K}}\hat{\psi}\p{k}
\Bar{\hat{T}\p{k}}\\
=\lim_{K\to\infty}\Bar{\sum_{\substack{k\in\Z^d\\\abs{k}\le K}}\hat{\Bar{\psi}}\p{-k}\hat{T}\p{k}}=\lim_{K\to\infty}\Bar{\sum_{\substack{k\in\Z^d\\\abs{k}\le K}}\hat{\Bar{\psi}}\p{-k}\br{T,\bf{e}_{-k}}}=\Bar{\br{T,\lim_{K\to\infty}\sum_{\substack{k\in\Z^d\\\abs{k}\le K}}\hat{\Bar{\psi}}\p{k}\bf{e}_k}}=\Bar{\br{T,\bar{\psi}}}\\=\br{\Bar{T},\psi}.
\end{multline}
This gives the sufficient condition for the second item.  To see that it is also necessary, suppose now that we have $T\in\mathcal{D}^\ast\p{\T^d;\C^\ell}$ satisfying $T=\Bar{T}$. We compute for $k\in\Z^d$:
\begin{equation}
    \Bar{\hat{T}\p{k}}=\Bar{\br{T,\bf{e}_{-k}}}=\Bar{\br{T,\Bar{\bf{e}_k}}}=\br{\Bar{T},\bf{e}_{k}}=\br{T,\bf{e}_k}=\hat{T}\p{-k}=R\hat{T}\p{k}.
\end{equation}
\end{proof}

\subsection{Fractional Sobolev spaces}

Here we record some results in fractional Sobolev spaces, as defined in Definition \ref{sobolev spaces}.

\begin{prop}\label{potential_map}
Let $s \in \R$.  Then the map $\Pi:  H^s_\| \p{\T^3;\R^3} \to \z{H}^{1+s}\p{\T^3;\R}$ defined by 
\begin{equation}
 \Pi f = \sum_{k\in \Z^3 \backslash \{0\}} \p{ \frac{k \cdot \hat{f}(k)}{2\pi i \abs{k}^2} } \bf{e}_k
\end{equation}
is well-defined, bounded, and linear, and $\nabla \Pi f = f$ for all $f \in  H^s_\| \p{\T^3;\R^3}$.  Moreover, $\Pi$ extends to a bounded linear map from $L^2\p{\R^+;H^s_\| \p{\T^3;\R^3}}$ to  $L^2\p{\R^+;\z{H}^{1+s}\p{\T^3;\R}}$ with the same properties.
\end{prop}
\begin{proof}
First we compute
\begin{equation}
 \norm{\Pi f}_{\z{H}^{1+s}}^2 = \sum_{k \in \Z^3 \backslash \{0\}} \abs{k}^{2+ 2s} \abs{ \frac{k \cdot \hat{f}(k)}{2\pi i \abs{k}^2} }^2 \le \frac{1}{4\pi^2} \sum_{k \in \Z^3 \backslash \{0\}} \abs{k}^{2s} \abs{  \hat{f}(k) }^2 = \frac{1}{4\pi^2}  \norm{ f}_{\z{H}^{s}}^2,
\end{equation}
and then we use Lemma \ref{real_lemma} to see that
\begin{equation}
 \Bar{\widehat{\Pi f}(k)} =  \frac{-k \cdot \Bar{\hat{f}(k) } }{2\pi i \abs{k}^2 } = \frac{-k \cdot \hat{f}(-k)  }{2\pi i \abs{-k}^2 } = \widehat{\Pi f}(-k),
\end{equation}
which then implies that $\Pi f$ is a real-valued distribution.  From this we deduce that $\Pi$ is a well-defined bounded linear map.  Then for  $f \in  H^s_\| \p{\T^3;\R^3}$ we compute
\begin{equation}
\widehat{\nabla \Pi f}(k) = 2\pi i k  \frac{k \cdot \hat{f}(k)}{2\pi i \abs{k}^2} = \f{k \otimes k}{\abs{k}^2} \hat{f}(k),
\end{equation}
to deduce that $\nabla \Pi f = (I- \P) f =f$.  The extension result then follows trivially.

\end{proof}

Next we record a useful product estimate in fractional Sobolev spaces.

\begin{prop}\label{products}
Let $d\in\N^+$, and suppose that $s,t\in\R^+\cup\cb{0}$ satisfy $s>\f{d}{2}$ and $s\ge t$.  Then there exists a constant $C >0$, depending on $s$ and $t$, such that if $f\in H^s\p{\T^d}$ and $h\in H^t\p{\T^d}$, then $fh\in H^t\p{\T^d}$ and 
\begin{equation}
 \norm{fh}_{H^t}\le C\norm{f}_{H^s}\norm{h}_{H^t}.
\end{equation}
\end{prop}
\begin{proof}
Since $s > d/2$ we have that $H^s(\T^d)$ is an algebra and that $H^s(\T^d) \hookrightarrow C^0_b(\T^d)$.  Consequently, for a fixed $f \in H^s(\T^d)$, if we define the linear operator $T_f$ via $T_f g = fg$, then $T_f$ is bounded linear map from $L^2\p{\T^d}$ to $L^2\p{\T^d}$ and from $H^s\p{\T^d}$ to $H^s\p{\T^d}$ satisfying
\begin{equation}
 \norm{T_f}_{\mathcal{L}(L^2)} \le \norm{f}_{C^0_b} \lesssim \norm{f}_{H^s} \text{ and } \norm{T_f}_{\mathcal{L}(H^s)}  \lesssim \norm{f}_{H^s}.
\end{equation}
If $s=t$, then we're done.  Otherwise $0 < t < s$, and so the theory of complex interpolation (see, for instance, Theorem 6.23 in~\cite{folland_1}) implies that $T_f$ is a bounded linear map from $H^t(\T^d)$ to $H^t(\T^d)$ and that there exists $\theta \in (0,1)$, depending on $t,s$, such that
\begin{equation}
 \norm{T_f}_{\mathcal{L}(H^t)} \le C  \norm{T_f}_{\mathcal{L}(L^2)}^{\theta}  \norm{T_f}_{\mathcal{L}(H^s)}^{1-\theta} \le C  \norm{f}_{H^s}
\end{equation}
for a constant $C>0$ depending only on $s,t$.  Then for $h \in H^t(\T^d)$ we have that $fh = T_f h \in H^t(\T^d)$ with the stated estimate.
\end{proof}

\subsection{Space-time Sobolev spaces}
 
Here we consider some useful embedding  properties of the space-time Sobolev spaced defined in Definition \ref{time into banach}.

\begin{prop}\label{interpolation}
Suppose that $r\in\R$, $n\in\N$, and $I= (0,T)$ for $T \in (0,\infty]$.  Then following hold.
\begin{enumerate}
\item For every $\N\ni k\le n$ and for all $f\in L^2\p{I;H^{2+r}\p{\T^3;\R^3}}\cap H^{n+1}\p{I;H^{r-2n}\p{\T^3;\R^3}}$ we have that $f\in H^k\p{I;H^{2+r-2k}\p{\T^3;\R^3}}$.  Moreover, there exists a constant $C$ independent of $f$ such that 
\begin{equation}
\norm{f}_{H^kH^{2+r-2k}}\le C\p{\norm{f}_{L^2H^{2+r}}+\norm{f}_{H^{n+1}H^{r-2n}}}. 
\end{equation}

\item If $\N\ni k\le n$, then for every $f\in L^2\p{I;H^{2+r}\p{\T^3;\R^3}}\cap H^{n+1}\p{I;H^{r-2n}\p{\T^3;\R^3}}$ we may redefine $f$ on a null set to arrive at the inclusion $f\in UC^k_b\p{I;H^{r-2k+1}\p{\T^3;\R^3}}$.  Moreover, there exists a constant $C$ depending on only $k$ such that 
\begin{equation}
\norm{f}_{C^{k}_bH^{r-2k+1}}\le C\p{\norm{f}_{L^2H^{2+r}}+\norm{f}_{H^{n+1}H^{r-2n}}}. 
\end{equation}
Finally, if $T=\infty$, then we have $\lim_{t\to\infty}\norm{f(t)}_{H^{r-2k+1}}=0$.
\end{enumerate}
\end{prop}
\begin{proof}
This is proved in Theorem 2.3 and 3.1 of \cite{lions_magenes_1}.
 
\end{proof}

\bibliographystyle{alpha}
\bibliography{mpr_biblio.bib}

\end{document}